\documentclass[11pt, reqno]{amsart}
\usepackage{etoolbox, mathtools}
\makeatletter
\patchcmd{\@settitle}{\uppercasenonmath\@title}{}{}{}
\patchcmd{\@setauthors}{\MakeUppercase}{}{}{}
\makeatother

\title{Sensitivity of steady states in a degenerately-damped
  stochastic Lorenz system} 
\author[J. F\"{o}ldes, N. E. Glatt-Holtz,  D. P. Herzog]
     {Juraj F\"{o}ldes, Nathan E. Glatt-Holtz,  David P. Herzog\\
  \scriptsize{emails: foldes@virginia.edu, 
                negh@tulane.edu, dherzog@iastate.edu}
}

\usepackage[margin=.9in]{geometry}
\usepackage{amsfonts, amssymb, amsmath, amsthm,enumerate,esint,multicol}
\usepackage[colorlinks=true, pdfstartview=FitV, linkcolor=blue,
            citecolor=blue, urlcolor=blue]{hyperref}
\usepackage[usenames]{color}
\definecolor{Red}{rgb}{0.7,0,0.1}
\definecolor{Green}{rgb}{0,0.7,0}
\usepackage{accents}
\usepackage{comment}
\usepackage{graphicx}
\usepackage[capitalize,nameinlink,noabbrev]{cleveref}
\usepackage{autonum}

\usepackage{bbm}
\definecolor{labelkey}{rgb}{0,0,1}

\numberwithin{equation}{section}

\newtheorem{Theorem}{Theorem}[section]
\newtheorem{Proposition}[Theorem]{Proposition}

\newtheorem{Definition}[Theorem]{Definition}
\newtheorem{Remark}[Theorem]{Remark}
\newtheorem{Question}[Theorem]{Question}

\newcommand{\PP}{\mathbb{P}}
\newcommand{\LL}{\mathcal{L}}
\newcommand{\B}{\mathcal{B}}
\newcommand{\N}{\mathbb{N}}
\newcommand{\M}{\mathcal{M}}

\newcommand{\indFn}[1]{1 \! \! 1_{#1}}
\newcommand{\E}{\mathbb{E}}
\newcommand{\Prb}{\mathbb{P}}
\newcommand{\RR}{\mathbb{R}}

\newcommand{\ad}{\mbox{ad}}
\newcommand{\spn}{\mbox{span}}
\newcommand{\ndeg}{\mathfrak{n}}

\begin{document}

\begin{abstract}
  We study stability of solutions for a randomly driven and
  degenerately damped version of the Lorenz '63 model.  Specifically,
  we prove that when damping is absent in one of the temperature
  components, the system possesses a unique invariant probability
  measure if and only if noise acts on the convection variable.  On
  the other hand, if there is a positive growth term on the vertical
  temperature profile, we prove that there is no normalizable
  invariant state.  Our approach relies on the derivation and analysis
  of non-trivial Lyapunov functions which ensure positive recurrence
  or null-recurrence/transience of the dynamics.

  \medskip {\noindent \scriptsize{ {\it {\bf Keywords:} Stochastic
        Differential Equations, The Lorenz '63 System,
        Invariant Measures, Noise Induced Stabilization, Lyapunov Construction}  \\
      {\it {\bf MSC2020:} 37H30, 37A25, 37A30, 37A50}}}
\end{abstract}

\maketitle

\setcounter{tocdepth}{1}
\tableofcontents

\newpage

\section{Introduction}

In this manuscript, we study the existence of invariant measures of the 
stochastic Lorenz system
\begin{align}
dx&= \sigma (y-x) \,dt + \sqrt{2\gamma_1} \, dB_1 \,,
	\notag\\
dy&= x(\rho-z) \, dt - y\, dt + \sqrt{2\gamma_2} \, dB_2 \,,
	\label{eqn:lor63}\\
dz&= xy \, dt - \beta z \,dt + \sqrt{2\gamma_3} \, dB_3  \,,
	\notag
\end{align}
where $B_i$, $i = 1, 2, 3$ are independent, standard Brownian motions
and $\sigma, \rho, \beta, \gamma_i$ are constants.  We assume that
$\sigma >0$ and $\rho \geq 0$, while for the diffusion parameters
$\gamma_1, \gamma_2, \gamma_3 \geq 0$ we require $\gamma_i>0$ for at
least one index $i$, which means that the system is genuinely
stochastic.  If $\beta > 0$, it is well known that \eqref{eqn:lor63}
possesses a normalizable invariant measure and the long term dynamics
has been extensively studied.  In the present paper we focus on a
`degenerate damping factor' $\beta \leq 0$, and we investigate whether
the presence of noise plays a nontrivial role in stabilizing the
dynamics.

\subsection*{Previous Literature}
The deterministic version of equation~\eqref{eqn:lor63}; that is, when
$\gamma_i = 0$ for $i = 1, 2, 3$, has a long history as a canonical
example of a chaotic dynamical system.  Originally \eqref{eqn:lor63}
was derived from the Boussinesq approximation of Rayleigh-B\' enard
convection, \cite{lorenz1963deterministic}.  It is understood as a
projection of the Boussinesq equation onto one Fourier direction with
wavenumber $k$, in which case $x$ represents the convection rate, and
$y$ and $z$ describe the horizontal and vertical temperature
variations, respectively.  In this framing as a simple model for
convection, $\sigma$ corresponds to the Prandtl number, $\rho$ is a
rescaled Rayleigh number and $\beta$ is an aspect ratio depending on
$k$.

While certainly $\beta$ is a strictly postive parameter in the
original derivation of \eqref{eqn:lor63}, if the Rayleigh number is
large, a typical assumptions for turbulent flows, and $k$ is large,
then $\beta \approx 0$.  Thus it is natural to investigate the system
with $\beta = 0$. On the other hand, practical numerical
considerations lead to the so-called Homogeneous Raylegh-B\' enard
(HRB) system, where a linearly unstable term, an analogue of the case
when $\beta \leq 0$, appears in the temperature equation.  See, for
example, \cite{CalzavariniEtAl2006} and a related two-dimensional ODE
stochastic model in \cite{BD_12}.  Furthermore, equations with similar
structure to HRB also appear in a certain zero Prandlt limit which
models mantle convection, see for example \cite{thual1992zero,
  scheel2017predicting}.  Thus both HRB and the zero Prandlt limit
provide additional motivation for studying the parameter range
$\beta \leq 0$ in \eqref{eqn:lor63}.

Here it is worth emphasizing that noise must be introducted in
\eqref{eqn:lor63} for there to be any hope that this system would
posses any (globally) stable statistics when $\beta \geq 0$.  Indeed,
in the absence of noise when $\beta <0$, the system~\eqref{eqn:lor63}
has initial conditions ($x_0 = y_0= 0$, $z_0 \neq 0$) leading to
infinite time blow-up. On the other hand if $\beta = 0$, then all
points on $z$-axis are equilibria, and therefore there is no compact
global attractor.  Nevertheless, in both cases, the set of initial
conditions that lead to blow-up (or equilibria) sit on a lower
dimensional subset of the phase space.  One may therefore inquire if
one can find suitable noise perturbations which kick trajectories off
of these meager subset of the phase space stabilizing the dynamics and
leading formation of statistically steady states.

Thus the topics studied in this paper for $\beta \leq 0$ fall into a
larger class of ``stabilization-by-noise" problems.  Such problems
have been investigated by a number of researchers in a variety of
contexts.  Let us next briefly recall those works closely related to
our setting.  Motivated by convection models in
\cite{hughes1990chaos,hughes1990low}, the effect of additive noise on
unbounded solutions was studied.  From another perspective advocated
recently in \cite{Elgindi2017}, the range $\beta \leq 0$ above
provides a turbulence analogue of a class of core models in
non-equilibrium statistical mechanics describing coupled oscillators
with heat baths at different temperatures~\cite{BT_02,EH_00, CEHB_18,
  Hairer_09,HaiM_09}.  Similar to such works on heat baths, one
associates a natural energy functional with \eqref{eqn:lor63} which is
``approximately conserved" but which is not globally dissipative.  In
particular, dissipation naturally acts on the $x$ and $y$ directions,
but not necessarily on the $z$-direction, unless of course $\beta >0$.
However, when $\beta \leq 0$, either there is no explicit dissipation
($\beta =0$) or there is in fact a source of linear instability
($\beta < 0$), so it is unclear whether the dissipation in $x$ and
$y$, coupled with the noise, can conspire to ``spread" the dissipation
to the $z$-direction.  Let us finally mention that it is known that an
arbitrary small additive noise can avert deterministic finite-time
blow-up and lead to stable dynamics, see for example~\cite{Sch_93,
  GHW_11, BHW_12, AKM_12, HerzogMattingly2015, LS_17, CFK_17}.  The
presence of noise can also induce stable oscillations~\cite{BD_12}
among other behaviors~\cite{KCSW_19}.

\subsection*{Statement of the Main Results}

In view of the above discussions we aim to answer the following
question in this paper:
\begin{Question}\label{q:IM}
  For what values of $\beta \leq 0$ and of
  $\gamma_1, \gamma_2, \gamma_3 \geq 0$ does \eqref{eqn:lor63} possess
  an invariant probability measure?
\end{Question}

\noindent Recall that in our context there is at least one index $i$
such that $\gamma_i>0$.  The answer to this question is known to be
affirmative for $\beta >0$ and, in fact, in this case the system is
geometrically ergodic when $\gamma_1>0$ and either $\gamma_2 >0$ or
$\gamma_3 >0$; see, for example, \cite{MSH_02}. Thus, our focus in
this paper is on the case when $\beta \leq 0$, where the associated
deterministic dynamics does not possess a compact global attractor.

Let us now present the main results in the paper concerning the
stochastic stability of \eqref{eqn:lor63} which addresses most of
\cref{q:IM}.
\begin{Theorem}
\label{thm:main}
Consider \eqref{eqn:lor63} and assume $\sigma > 0$, $\rho \geq 0$ and
$\gamma_1, \gamma_2, \gamma_2 \geq 0$ with at least one index $i$ for
which $\gamma_i>0$.  For any value of $\beta \leq 0$, the stochastic
dynamics is globally defined (non-explosive in the sense of
\eqref{eq:no:blow} below).
\begin{itemize}
\item[(i)] If $\beta = 0$, $\gamma_1 > 0$, then \eqref{eqn:lor63} has
  a unique invariant probability measure.
\item[(ii)] If $\beta =0$, $\gamma_1 =0$, and one of
  $\gamma_2, \gamma_3$ is positive, then \eqref{eqn:lor63} does not
  possess an invariant probability measure.
\item[(iii)] Finally if $\beta < 0$, then for any
  $\mathcal{K}\subseteq \mathbb{R}^3$ compact, there exists
  $(x,y,z) \notin \mathcal{K}$ such that
\begin{align}
\E_{(x,y,z)} \xi_\mathcal{K}=\infty
\end{align}
where 
\begin{align}
\label{eqn:firsthit}
\xi_\mathcal{K}=\inf\{ t \geq 0 \, : \, (x_t, y_t, z_t ) \in \mathcal{K}\}.  
\end{align}
Consequently, if we furthermore assume that $\gamma_1 > 0$ and either
$\gamma_2 > 0$ or $\gamma_3 > 0$, then \eqref{eqn:lor63} does not
possess an invariant probability measure.
\end{itemize}
\end{Theorem}

\begin{Remark}
  Depending on which noise parameters $\gamma_i$ are positive, the
  issue of uniqueness of invariant measures for the
  system~\eqref{eqn:lor63} can also be subtle.  In the recent
  interesting paper~\cite{CH_20}, it is shown that when $\beta >0$,
  $\gamma_1=\gamma_2=0$ and $\gamma_3 >0$, then invariant measures can
  either be unique or not, and the uniqueness depends on the magnitude
  of the non-zero noise parameter $\gamma_3>0$.
\end{Remark}

\subsection*{Overview of the Analysis}

Since the coefficients of equation~\eqref{eqn:lor63} are globally
smooth ($C^\infty$) functions, the proof of well-posedness of
\eqref{eqn:lor63} follows immediately once one establishes absence of
finite-time explosion.  In our case, non-explosivity is then concluded
by using the natural Lyapunov function associated to the dynamics.
See \cref{prop:noexp} for further details.  However, since
$\beta \leq 0$, this natural function is not robust enough to
determine the existence/non-existence of an invariant probability
measure precisely because of the absence of explicit dissipation in
the $z$-direction.  Thus we cannot use this function directly to
answer our main question.

To this end, the typical route used to conclude
existence/non-existence of an invariant probability measure is to
estimate, for a `big' compact set $\mathcal{K}\subseteq \mathbb{R}^3$,
the expectation of the random variable $\xi_\mathcal{K}$ as
in~\eqref{eqn:firsthit}.  Indeed, if one can show that
$(x,y,z) \mapsto \E_{(x,y,z)} \xi_\mathcal{K}$\footnote{$\E_{(x,y,z)}$
  denotes the expectation with respect to the process $(x_t,y_t, z_t)$
  started at $(x,y,z)$ at time $0$} is bounded on compact sets in
$\mathbb{R}^3$, then an invariant probability measure can be
constructed using a slight modification of the cycle argument of
Khasminskii~\cite{Khas_11}.  See also \cite{RB_06, Kliem_87}.  On the
other hand, if there is sufficient noise in the
system~\eqref{eqn:lor63} by way of hypoellipticity and support
properties of the solution, then global finiteness of the function
$\E_{(x,y,z)} \xi_\mathcal{K}$ for some $\mathcal{K}$ compact is
equivalent to the existence of an invariant probability
measure~\cite{Kliem_87}.  Thus our arguments center around determining
whether or not this expectation can be shown to be globally finite.

Our approach to estimating $\E_{(x,y,z)}\xi_\mathcal{K}$ relies on
detailed Lyapunov constructions.  Specifically, to show the
expectation is bounded on compact sets, we seek a $C^2$ function
$V: \RR^3 \to \RR$ with $V(x,y,z) \rightarrow \infty$ as
$|(x,y,z)|\rightarrow \infty$ and such that
\begin{align}
  \mathcal{L}V \leq -c + d \indFn{\mathcal{K}} 
  \label{eq:rec:cond:Lyp}
\end{align}
for some $c, d > 0$ and some compact set $\mathcal{K} \subset \RR^3$,
where $\mathcal{L}$ is the infinitesimal generator of
\eqref{eqn:lor63} defined in~\eqref{eq:Gen:L63:gen:case}.  While the
above Lyapunov criteria is quite standard to show existence of an
invariant probability measure (see, for example, \cite{Khas_11, MT_93,
  RB_06}), we will see that Lyapunov constructions can also be
employed to establish nonexistence.  Following \cite{Won_66} and
generalizations more recently appearing in \cite{Hairer_09, Khas_11},
\cref{thm:criterianon} below identifies a condition guaranteeing
infinite expected return times depending on the existence of two test
functions, $V_1,V_2:\RR^3\rightarrow \RR$ satisfying Lyapunov-like
conditions.

Regarding the existence of an invariant probability measure when $\beta = 0$, we
construct a suitable Lyapunov function $V$ satisfying
\eqref{eq:rec:cond:Lyp} by pivoting off of the
 natural Lyapunov function one uses to show well-posedness, namely 
\begin{align}
  H(X) = H(x,y,z):= x^2+y^2+z^2- 2(\sigma+ \rho)z.  
  \label{eq:lyp:base:intro}
\end{align}
Here, a direct computation, see \eqref{eq:Gen:L63:gen:case}, leads to
\begin{align}
  \mathcal{L} H = -2(\sigma x^2+y^2) - 2\beta z^2 + 2\beta (\sigma + \rho) z
                   + 2( \gamma_1 + \gamma_2 + \gamma_3) 
  \label{eq:gen:lyp:b:intro}
\end{align}
which reveals the necessity of adding supplemental terms to $H$ thus
seeking a Lypunov function of the form $H + \psi$ in order to achieve
\eqref{eq:rec:cond:Lyp} for the region where $x,y$ are bounded but
where $|z|$ is large; that is \eqref{eq:rec:cond:Lyp} requires
$\mathcal{L} (H + \psi)$ to be uniformly negative off of a compact
set.  The definition of the supplemental perturbation $\psi$ makes use
of certain rescalings of the dynamics at large values, allowing one to
better parse relevant terms in a neighborhood of the point at
infinity.  Note that these asymptotics initially yield a piecewise
definition of the perturbation $\psi$ which, in order to obtain a
globally $C^2$ requires that the different regions be glued together.
The detailed derivation of $\psi$ and the motivation behind the
scaling we choose are carried out in \cref{sec:exist}.

To extract conclusion (iii) in \cref{thm:main}, we again employ
Lyapunov methods in order to estimate the expected return time to a
given compact set $\mathcal{K}$.  The principal observation leading to
the proof is that the function $M$ given by
\begin{align}
  M(X) = M(x,y,z) = 2 \sigma z - x^2
  \label{eq:z:cru:test}
\end{align}
satisfies 
\begin{align}
  \mathcal{L} M = 2 \sigma( | \beta| z + x^2) - \gamma_1 \,.  
  \label{eq:gen:M:test}
\end{align}
Now, if it were the case (although it is far from true) that the set
$\{ (x,y,z) \, : \sigma z \geq x \}$ is invariant for the dynamics,
then~\eqref{eq:z:cru:test} and~\eqref{eq:gen:M:test} would together
imply that the solution is growing exponentially in this region,
provided $z>0$ is large enough initially.  The proof then nontrivially
modifies this initial observation to conclude the result, even in the
presence of noise or dynamics effects that can steer the process in
and out of this region.  See ~\cref{sec:non-existence} for further
details.

Finally, to treat the borderline case $\beta = 0$ when $\gamma_1 = 0$,
we proceed with a direct approach.  For example, when
$\gamma_1 =\beta=0$ but either $\gamma_2$ or $\gamma_3$ are strictly
positive we again use of the test function $M$ defined in
\eqref{eq:z:cru:test}.  At least formally, \eqref{eq:gen:M:test} and
Dynkin's formula immediately implies that $\E x \equiv 0$ if the
initial condition is distributed as an invariant state.  However, this
produces a contradiction to the structure of \eqref{eqn:lor63}, since if
$\gamma_2 > 0$, then $y$ is non-trivial, leading to non-zero
derivative of $x$. On the other hand, if $\gamma_3 > 0$ and $y = 0$,
then $z$ evolves as a Brownian motion, which is not a normalizable
invariant.  Similar direct argumentation can be employed to show that
when $\gamma_2 = \gamma_3 =\beta = 0$ and $\gamma_1 > 0$, then the
only stationary solution is an Ornstein-Uhlenbeck process concentrated
on the $x$ component.

\subsection*{Organization} The rest of the manuscript is organized as
follows.  \cref{sec:methods} provides a self contained summary of some
general results on Lyapunov methods for It\^{o} diffusions.
\cref{sec:uniqueness:Lor} contains the details for nonexplosivity and
conditions for the uniqueness of invariant measures for
\eqref{eqn:lor63}.  The results of this section also imply the
uniqueness part of \cref{thm:main}~(i) when $\gamma_1>0$ and
either $\gamma_2 >0$ or $\gamma_3 >0$. In \cref{sec:exist}, we carry
out the construction and rigorous analysis of a Lyapunov function
leading to the existence of an invariant probability \eqref{eqn:lor63}
in the case when $\beta = 0$ and $\gamma_1 > 0$.  The main result in
this section, \cref{prop:pos:rec}, in particular establishes the
existence part of \cref{thm:main}~(i).  In \cref{sec:sensitive:noise},
we prove \cref{thm:main} (ii) and establish uniqueness of the
invariant measure when $\beta =0$, $\gamma_1 >0$ and
$\gamma_2=\gamma_3=0$, thus finalizing the proof of
\cref{thm:main}~(i).  Finally, in \cref{sec:non-existence} we
prove \cref{thm:main} (iii).  This proof also relies on
Lyapunov constructions.

\section{Methodological Foundations: 
  Lyapunov Techniques}
\label{sec:methods}

This section presents some general results on Lyapunov methods for
It\^o diffusions, providing the foundation for our analysis in later
sections.  To keep the paper self-contained, we present detailed
proofs for some results familiar to specialists, but which are
dispersed in literature under varied sets of assumptions.

Let $M_{nk}$ denote the set of $n \times k$ matrices with entries in
$\RR$.  Given any $F \in C^2(\RR^n; \RR^n)$ and
$G = (G_1, \ldots, G_k) \in C^2(\RR^n; M_{nk})$, let $X_t$ be the
process on $\RR^n$ satisfying the It\^{o} stochastic differential
equation
\begin{align}
\label{eqn:sdegen}
dX_t= F(X_t) \, dt + G(X_t) \, dB_t
    =F(X_t) \, dt + \sum_{l = 1}^k G_l(X_t) \, dB_t^l\,.
\end{align}
Here, $B_t=(B_t^1, \ldots, B_t^k)^T$ is a standard $k$-dimensional
Brownian motion defined on a filtered probability space
$(\Omega, \mathcal{F}, \{\mathcal{F}_t\}_{t \geq 0}, \PP)$, where $\E$
denotes the corresponding expected value.  We denote by $\LL$ the
infinitesimal generator of the process $X_t$ acting on functions
$V\in C^2(\RR^n; \RR)$, namely,
\begin{align}
\LL V(X) :=& F(X) \cdot \nabla V(X) 
             + \frac{1}{2} (G G^T)(X) \nabla^2 V(X)
            \notag\\
  =& \sum_{j = 1}^n F_j(X)  \partial_{X_j} V(X) 
     + \frac{1}{2}\sum_{i,j =1}^n \sum_{l =1}^k G_{il} G_{jl}(X) 
              \partial^2_{X^i X^j} V(X). 
\label{fpe}
\end{align}    
Let $\B$ be the Borel $\sigma$-field of subsets of $\RR^n$.

Although globally defined solutions~\eqref{eqn:sdegen} in time are not
guaranteed for general $C^2$ drifts $F$ and diffusions $G$, there are
unique pathwise solutions defined until the first time $\tau$ in which
the process leaves every bounded domain in space.  Specifically, given
a fixed initial condition $X$, if we define stopping times
\begin{align}
  \tau_n := \inf\{ t\geq 0 \, : \, |X_t| \geq n\}, 
  \,\, \text{ and } \,\, 
  \tau  := \lim_{n\rightarrow \infty} \tau_n , 
  \label{eq:announce:blow}
\end{align}
then 
solutions exist and are unique for all times $t<\tau$, $\PP$-almost
surely.  We call $\tau$ the \emph{explosion time} of the process $X_t$
and say that $X_t$ is \emph{nonexplosive} if
\begin{align}
\PP_X\{ \tau < \infty\} =0 \,\,\text{ for all initial conditions}\,\, X\in \RR^n. 
    \label{eq:no:blow}
\end{align} 
In the above~\eqref{eq:no:blow}, the notation $\PP_X$ means the
process $X_t$ has $X_0=X\in \RR^n$.

If $X_t$ is nonexplosive, solutions of~\eqref{eqn:sdegen} exist and
are unique for all times $t\geq 0$ almost surely. Moreover, $X_t$
generates a Markov process and we define the transition probability
measure $\mathcal{P}_t$ as
$\mathcal{P}_t(X, \,\cdot\, )= \PP_X\{ X_t \in \,\cdot\, \}$.  The
Markov semigroup $\mathcal{P}_t$ acts on bounded, $\B$-measurable
functions $V: \RR^n\rightarrow \RR$ via
\begin{align}
  \mathcal{P}_t V(X) = \E_X V(X_t)
  = \int_{\RR^n} V(Y) \mathcal{P}_t(X,dY),
  \quad X \in \RR^n,
\end{align}
and on borel measures $\pi$ according to
\begin{align}
  \pi \mathcal{P}_t(A) = \int_{\RR^n} \pi(dX) \mathcal{P}_t(X, A)
  \quad A \in \B.
\end{align}
We say that a positive measure $\pi$ is an \emph{invariant measure}
for $\mathcal{P}_t$ if $\pi \mathcal{P}_t= \pi$ for all $t\geq 0$.  An
invariant measure $\pi$ for $\mathcal{P}_t$ with $\pi(\RR^n)=1$ is
called an \emph{invariant probability measure} for $\mathcal{P}_t$.

The next result outlines the criteria for a process defined by
\eqref{eqn:sdegen} to be both nonexplosive and have finite expected
returns to a compact set.

\begin{Proposition}
\label{prop:stab}
Given $F, G \in C^2$, the following statements hold for solutions
$X_t$ of \eqref{eqn:sdegen} and the corresponding infinitesimal
generator $\LL$ defied in \eqref{fpe}.
\begin{itemize}
\item[\emph{(a)}] Suppose that there exist a function
  $V\in C^2(\RR^n; \RR)$ such that $V(X) \rightarrow \infty$ as
  $|X|\rightarrow \infty$ and  constants $c, d>0$
  with the global bound 
  \begin{align}
    \LL V(X) \leq c V(X) + d \qquad \textrm{for all }\quad X \in \RR^n.  
    \label{eq:LV:does:not:blow}
  \end{align}
Then $X_t$ is nonexplosive, namely \eqref{eq:no:blow} holds.  
\item[\emph{(b)}] Suppose that $X_t$ is nonexplosive and that there
  exist a function $V\in C^2(\RR^n; [0, \infty))$, a compact set
  $\mathcal{K}\subseteq \RR^n$ and constants $c,d>0$ such that 
\begin{align}
  \LL V(X)\leq - c + d\textbf{1}_{\mathcal{K}}(X)
  \qquad \textrm{for all }\quad X \in \RR^n. 
  \label{eq:cru:Lyp:bound}
\end{align}   
 If
\begin{align}
  \xi_\mathcal{K} :=\inf \{ t\geq 0 \, : \, X_t \in \mathcal{K}\}  
  \label{eq:return:tm:comp}
\end{align}
denotes the first hitting time of $\mathcal{K}$ by $X_t$, then
\begin{align}
\label{eqn:expbound}
\E_X \xi_{\mathcal{K}} \leq \frac{V(X)}{c},   
\end{align}
for all $X\in \RR^n$.
\end{itemize}
\end{Proposition}

The proof of the proposition above is a straightforward application of
It\^{o}'s formula and can be found in a number of references, see, for
example~\cite{Khas_11, MT_93, RB_06}.  To illustrate the basic idea,
we provide details for part (b).

\begin{proof}[Proof of \cref{prop:stab} \emph{(b)}]
  Take
  $\xi := \xi_{t,n, \mathcal{K}} := t\wedge \tau_n \wedge
  \xi_\mathcal{K}$ with $\tau_n$ defined as in
  \eqref{eq:announce:blow}.  If $X \in \mathcal{K}$, then
  $\E_X \xi_{\mathcal{K}} = 0$ and \eqref{eqn:expbound} follows.
  Otherwise, $\textbf{1}_{\mathcal{K}}(X) = 0$ and by Dynkin's formula
  and \eqref{eq:cru:Lyp:bound} we have
\begin{align} 
    0 \leq \E_X V(X_{\xi}) 
         = V(X) + \E_X \int_0^{\xi} \LL V(X_s) \, ds 
  \leq V(X) - c \E_X \xi,  
\end{align}  
for any $t \geq 0$ and $n \geq 1$.  Rearranging and using
$V\geq 0$ produces the estimate
\begin{align}
\E_X \xi_{t,n, \mathcal{K}} \leq \frac{V(X)}{c}.  
\end{align}
Passing $n\rightarrow \infty$ and then $t\rightarrow \infty$ using both
monotone convergence and nonexplosivity,  that is, $\tau_n \to \infty$
a.s., gives the desired bound \eqref{eqn:expbound}.
\end{proof}
 
The next result provides the criteria we use in
\cref{sec:non-existence} to show that the expected return time to
compact sets is infinite for some initial data in the case case when
the parameter $\beta <0 $ in equation~\eqref{eqn:lor63}.  While the
result presented here can be traced back to at least \cite{Won_66} we
believe it deserves further attention as a powerful tool for the study
of stochastic (in)stability.  Note that the original formulation in
\cite{Won_66} imposes more hypotheses on the process $X_t$ than
needed; for example a uniform ellipticity assumption for the
generator $\LL$ was imposed in \cite{Won_66}.  This was noticed in the
paper~\cite{Hairer_09}, where a generalization of the results from
\cite{Won_66} is stated.  Here, we provide the details for this
generalization and also phrase the conclusions in a slightly different
way.  See also Lemma~3.11 of~\cite{Khas_11}.

To formulate the result, for $R>0$ we let 
\begin{align}
\xi_R= \inf\{ t\geq 0 \, : \, |X_t| \leq R\},
  \label{eq:return:ballz}
\end{align}
that is, $\xi_R$ is the first hitting time of the closed ball of radius
$R>0$ centered at the origin in $\RR^n$.  This is a small abuse of
notation, see $\xi_\mathcal{K}$ in \eqref{eq:return:tm:comp} above,
but there should not be any confusion given the context.

\begin{Theorem}
\label{thm:criterianon}
Suppose that there exist $V_1, V_2 \in C^2(\RR^n; \RR)$ satisfying the
following properties:
\begin{itemize}
\item[\emph{(p1)}]
  $\displaystyle{\limsup_{|X|\rightarrow \infty} V_1(X) = \infty.}$
\item[\emph{(p2)}] $V_2$ is strictly positive outside of a compact
  set.
\item[\emph{(p3)}]
  $\displaystyle{\limsup_{S \rightarrow \infty}\frac{\max_{|X|=S}
      V_1(X)}{\min_{|X|=S} V_2(X)} =0} $.
\item[\emph{(p4)}] There exists $R>0$ such that 
  \begin{align}
    \LL V_1(X) \geq 0 \quad \text{ and } \quad \LL V_2(X) \leq 1 
    \label{eq:inf:gen:blow:cond}
  \end{align}
  for every $|X| \geq R$, where $\LL$ is the generator for
  \eqref{eqn:sdegen} given in \eqref{fpe}.
\end{itemize}  
Then, there exists $M \geq 0$ such that
\begin{align}
\E_{X_*} \xi_{R}= \infty, \,
  \quad \text{ whenever } |X_*| \geq R \;
  \text{ and } \; V_1(X_*) \geq M.
  \label{eq:inf:ret:tm}
\end{align}
\end{Theorem}

\begin{proof}
  First notice that, given $V_1,V_2 \in C^2(\RR^n; \RR)$ and $R$ satisfying
  \emph{(p1)-(p4)}, one can add a negative constant to $V_1$ to 
  obtain  
  \begin{align}
    V_1(X) \leq 0, \quad \text{ for every } |X| \leq R 
    \label{eq:v1:neg:inner}
  \end{align}
  and a positive constant to $V_2$  so that 
  \begin{align}
    V_2(X) \geq 0 \quad \text{ for every } X \in \RR^n.
    \label{eq:v2:alway:p}
  \end{align}
  Since an addition of constants does not affect \emph{(p1) -(p4)} we
  proceed assuming \eqref{eq:v1:neg:inner} and \eqref{eq:v2:alway:p}.

  Let us fix an arbitrary $|X_*| \geq R$ such that $V_1(X_*) > 0$.
  Invoking (p1) we can choose a sequence of points $x_k \in \RR^n$,
  $k = 1, 2, \cdots$ such that $x_1=X_*$, $R<|x_k| < |x_{k+1}|$ for
  all $k$ and such that $0 < V_1(x_k) \uparrow \infty$ as
  $k \rightarrow \infty$.  Let
  $\tau'_k = \inf\{ t\geq 0 \, : \, |X_t| \geq |x_k|\}$ and recalling
  \eqref{eq:return:ballz} define functions $u_k(X,t)$ by
  \begin{align}
    u_k(X, t) = \E_X (\tau_k' \wedge \xi_{R}\wedge t).  
  \end{align}        
  Note that, in particular, since $x_k$ is an increasing sequence, we
  have the relationship
  \begin{align}
    0 \leq u_k(X,t) \leq u_{k+1}(X, t) 
  \end{align}
  for all $X$ with $R\leq |X| \leq |x_k|$ and all $k \in \N$,
  $t\geq 0$.  

  Define $M_k = \max_{|Y|=|x_k|} V_1(Y)$.  Note that $M_1>0$ and by
  passing to the relevant subsequence of $x_n$'s via (p1) we can
  assume that $M_{k + 1} > M_k$ for all $k$.  Let
  \begin{align}
    \lambda_k := M_{k}^{-1}\min_{|Y|=|x_k|} V_2(Y) 
               = \frac{\min_{|Y|=|x_k|} V_2(Y)}{\max_{|Y|=|x_k|} V_1(Y)}
    \label{eq:lamb:def:n}
  \end{align}
  and consider the functions
  \begin{align}
    V(k, X) := \lambda_k V_1(X) - V_2(X) =
    \frac{\min_{|Y|=|x_k|} V_2(Y)}{\max_{|Y|=|x_k|} V_1(Y)}
    V_1(X) - V_2(X),
    \label{eq:Vn:def}
  \end{align}
  for each $k \in \N$.  In view of assumption (p3), after passing to a
  subsequence, and recalling our choice of $X^*$ such that
  $V_1(X^*) > 0$ we have that
  \begin{align}
    \lim_{k \to \infty} \lambda_k = \infty
    \quad \text{ and } \lim_{k \to \infty} V(k, X^*) = \infty.
    \label{eq:good:lamb}
  \end{align}
  Furthermore, with \eqref{eq:v1:neg:inner}, \eqref{eq:v2:alway:p}
  we see that $V(k, X)$ is nonpositive on the boundary of the annulus
  $\mathcal{A}_k := \{X\, : \, R < |X|< |x_k|\}$, namely
  \begin{align}
    V(k, X) \leq 0 
    \quad \text{ for every } 
    X \in \{ Y \in \RR^n : | Y | = R \text{ or } |Y| = |x_k|\},
    \label{eq:good:Vn}
  \end{align}
  for every $n$. 

  Next, by Dynkin's formula and then invoking (p4) to produces the
  final inequality, we have
  \begin{align}
    -\E_{X_*} V(k, X_{\tau_k' \wedge \xi_{R}\wedge t}) 
       &= - V(k, X_*) 
         - \E_{X_*} \int_0^{\tau_k' \wedge \xi_{R}\wedge t}  
               \LL V(k, X_s) \, ds \\
       &=  u_k(X_*, t)- V(k,  X_*)
         + \E_{X_*} \int_0^{\tau_k' \wedge \xi_{R} \wedge t}  
         (\LL(V_2- \lambda_k V_1)(X_s) -1) \, ds\\
       & \leq u_k(X_*,t) - V(k, X_*). 
  \end{align} 
  for any $k\in \N$, $t\geq 0$.  Note that if
  $\E_{X_*}( \xi_{R} \wedge \tau_k' )= \infty$ for some $k \in \N$,
  the desired result follows from the monotone convergence theorem
  after passing $k \to \infty$.  Thus, we are left with the case
  $\E_{X_*}(\xi_{R} \wedge \tau_{k}') < \infty$, and in particular
  $\Prb (\xi_{R} \wedge \tau_{k}' < \infty) = 1$ for all $k$.  For
  fixed $k$, $V(k, \cdot)$ is bounded and continuous on
  $\mathcal{A}_k$, and therefore by the monotone and dominated
  convergence theorem after passing $t \to \infty$ we obtain
  \begin{align}
    - \E_{X_*} V(k, X_{\tau_k' \wedge \xi_{R}})     
      \leq \E_{X_*}(\xi_{R} \wedge \tau_k') - V(k, X_*) \,.
  \end{align}
  Since $\xi_{R} \wedge \tau_{k}' < \infty$ is almost surely bounded,
  \eqref{eq:good:Vn} produces the inequality
  \begin{align}
    V(k, X_{*}) \leq \E_{X_*}(\xi_{R} \wedge \tau_k')
               \leq \E_{X_*}\xi_{R}
  \end{align}  
  valid for all $k\in \N$.  Thus invoking \eqref{eq:good:lamb} we
  conclude \eqref{eq:inf:ret:tm} completing the proof.
\end{proof}

In summary \cref{prop:stab} and \cref{thm:criterianon} provide a basis
for analyzing the expected return time to compact sets for general
diffusions of the form \eqref{eqn:sdegen}.  For our purposes here we
can then appeal to general results found in e.g \cite{Kliem_87, MT_12}
to conclude either the existence or the non-existence of an invariant
probability measure for $\mathcal{P}_t$.  Note however that, at this
step in the analysis, we further require that $\mathcal{P}_t$ maintain
certain support and regularity properties.

In order to restate the results from \cite{Kliem_87, MT_12}, we need
the following definitions.
\begin{Definition}
  \label{def:hypo:ellip}
  Suppose that $\mathcal{A}$ is a differential operator defined on an
  open subset $U\subseteq \RR^n$.  We say that $\mathcal{A}$ is
  \textbf{hypoelliptic} on $U$ if for any distribution $u$ defined on
  an open subset $V\subseteq U$ such that
  $\mathcal{A} u \in C^\infty(V)$, we have $u\in C^\infty(V)$.
\end{Definition}

\begin{Definition}
  \label{def:nice:diff}
  We say that $X_t$ satisfying~\eqref{eqn:sdegen} is \textbf{nice
    diffusion} if it is non-explosive as in \eqref{eq:no:blow} and the
  following conditions are met:
\begin{itemize}
\item[(i)]  $F\in C^\infty(\RR^n; \RR^n)$ and $G\in C^\infty(\RR^n ; M_{nk})$;
\item[(ii)] The operators
  $\LL, \LL^*, \LL\pm \partial_t, \LL^*\pm \partial_t$ are
  hypoelliptic on the respective domains
  $\RR^n, \RR^n, \RR^n\times(0, \infty), \RR^n\times(0, \infty)$ where
  $\LL^*$ denotes the formal adjoint of $\LL$ with respect to the
  $L^2(\RR^n; dx)$ inner product.
\item[(iii)] $\text{\emph{supp}}(\mathcal{P}_t(X, \, \cdot \,))=\RR^n$
  for all $t>0, X\in \RR^n$.\footnote{Recall that, given a probability
    measure $\mu$ on $\RR^n$,
  \begin{align}
   \text{supp}(\mu)
    := \{ x \in \RR^n: \mu(\{y : |y - x| < \epsilon\}) > 0, 
          \text{ for every } \epsilon > 0\}.   
  \end{align}
  In particular, $\text{supp}(\mu) = \RR^n$ if 
   $\mu$ is continuously distributed and its density is almost surely
  positive.
  }   
\end{itemize} 
\end{Definition}

Note that hypoellipticity of $\mathcal{A}$ on $U$ intuitively means
that $\mathcal{A}$ has a  local smoothing effect  on $U$ reminiscent
of elliptic operators.  Hypoellipticity of 
$\LL, \LL^*, \LL\pm \partial_t, \LL^* \pm \partial_t$ implies
their smoothing properties and in addition the probability density functions
of the associated stochastic differential equations exist and are 
smooth
 in all variables (forward, backward and time).  Furthermore if an invariant
probability measure exists, the hypoellipticity  guarantees
the existence and smoothness of an invariant probability density.
This is the reason we assume condition (ii) in Definition \ref{def:nice:diff}.

\begin{Proposition}
\label{prop:criteria}
Suppose that $X_t$ is a nice diffusion according to
\cref{def:nice:diff}.  Then we have the following:
\begin{itemize}
\item[(a)]  There is at most one invariant probability measure
  for $\mathcal{P}_t$.     
\item[(b)] $\mathcal{P}_t$ has an invariant probability measure if and
  only if there exists $R>0$ such that $\E_X \xi_{R} < \infty$ for all
  $X\in \RR^n$ and the mapping $X\mapsto \E_X\xi_R$ is bounded on
  compact subsets of $\RR^n$.  In the above, we recall that $\xi_{R}$
  is the return time defined in \eqref{eq:return:ballz}.
\end{itemize}
\end{Proposition}

The proof of \cref{prop:criteria} combines results scattered in the
literature; cf. \cite{MT_57, Khas_60, Kliem_87, RB_06, Khas_11,
  MT_12}.  Part (a) is a standard consequence of ergodic
decomposition, see, for example~\cite[Proposition~8.1]{RB_06}.  For
part (b), if there exists $R>0$ such that $\E_X \xi_R< \infty$ for all
$X\in \RR^n$ and the mapping $X\mapsto \E_X \xi_R$ is bounded on
compact subsets of $\RR^n$, the unique invariant probability measure
can be constructed using Khasminskii's cycle argument as
in~\cite{Khas_11, RB_06}.  The remaining implication in part (b) is
more subtle, as it relies on the dichotomy between transient points
and recurrent points for degenerate diffusions.  This was established
in~\cite{Kliem_87}.

We next recall a set of criteria which can be used to establish the
smoothness and positivity hypothesis of \cref{def:nice:diff} required
for \cref{prop:criteria}. First,  we formulate in our setting
\cite[Theorem 2.9]{HM_15i} which is a combination of H\"{o}rmander's
hypo-ellipticity theorem~\cite{Hor_67}, ensuring the existence and
smoothness of a density (with respect to Lebesgue measure on $\RR^n$),
with the support theorems of Stroock and
Varadhan~\cite{SV_72ii,SV_72i} relating positivity of the density to
controllability.  By the results from \cite{HM_15i},  for~\eqref{eqn:lor63}, one can use certain
Lie brackets as in~\cite{Hor_67} to obtain both the regularity of
the density and support of the transition measure.

To formulate the result let us introduce  preliminary
definitions and notations following as closely as possible the
formulations in \cite{HM_15i}.  Recall that the \emph{Lie
  bracket} of (smooth) vector fields
\begin{align}
  U(X) = \sum_{j = 1}^n U^j(X) \frac{\partial}{\partial X_j},
  \qquad
  W(X) = \sum_{j = 1}^n W^j(X) \frac{\partial}{\partial X_j},
  \label{eq:Vec:diff:geo}
\end{align}
is given as
\begin{align}
  [U, W](X) 
  = \sum_{j = 1}^n 
  \sum_{k = 1}^n \left(
      U^k(X) \frac{\partial W^j(X)}{\partial X_k}
      -  W^k(X) \frac{\partial U^j(X)}{\partial X_k}
     \right)  \frac{\partial }{\partial X_j}.
\end{align}
We then introduce, for any vector fields $U$ and $W$
\begin{align}
  \ad^0U(W) = W,
  \quad
  \ad^1U(W) = [U,W],
  \quad
  \ad^m U(W) := \ad^1U(\ad^{m-1}U(W)),
  \text{ for } m \geq 2.
\end{align}
When $W$ is a polynomial vector field (that is to say $W$ depends
polynomially on the components of $X$), for any $X \in \RR^n$ we
denote
\begin{align}
  \ndeg(X,W) := \max_{j = 1, \ldots, n} \mbox{deg}(p_j)
  \quad \text{ where } p_j(\lambda) := W_j(\lambda X).
  \label{eq:rel:deg:cond}
\end{align}
For any collection of vector fields $\mathcal{G}$
on $\RR^n$ we define
\begin{align}
  \mbox{cone}_{\geq 0}(\mathcal{G}) 
   = \left\{
      \sum_{j = 1}^N \lambda_j U_j :
  \text{ for any finite collections } \{\lambda_1, \ldots, \lambda_N\}
  \subset [0,\infty), \{U_1, \ldots, U_N\} \subset \mathcal{G}
  \right\}.
\end{align}

For simplicity and in the view of 
\eqref{eqn:sdegen} we restrict to the case when the diffusion coefficients $G$ are
independent of $X$ and the drift $F$ is a polynomial.  Let
\begin{align}
  \mathcal{G}_0 := \spn\{ G_1, \ldots, G_k \}
\end{align}
and starting at $j = 1$ we define\footnote{Note that in
  \eqref{eq:G1:brak} and \eqref{eq:Gjp1:brak} we treat constant vector
  fields $G$ as a vector in $\RR^n$ when computing $\ndeg(G,F)$.}
\begin{align}
  \mathcal{G}_1^{O} &:= \mathcal{G}_0  \cup
  \{ \ad^{\ndeg(G,F)} G(F) :  G \in \mathcal{G}_0, \ndeg(G, F) \text{ is odd} \}
  \notag\\
  \bar{\mathcal{G}}_1^{O} &:= \{G \in \mathcal{G}_1^{O}: G \text{ is a constant
                      vector field} \}
  \notag\\
  \mathcal{G}_1^{E} &:= 
  \{ \ad^{\ndeg(G,F)} G(F) :  G \in \mathcal{G}_0, \ndeg(G, F) \text{ is even} \}
  \notag\\
  \mathcal{G}_1 &= \spn(\mathcal{G}_1^{O}) 
                  +   \mbox{cone}_{\geq 0}(\mathcal{G}_1^{E}).
                  \label{eq:G1:brak}
\end{align}
We then proceed iteratively to define, for $j \geq 1$
\begin{align}
  \mathcal{G}_{j+1}^O &:= \mathcal{G}_{j}^O \cup
  \{ \ad^{\ndeg(G,F)} G(H) :  
       G \in \bar{\mathcal{G}}_j^O, H \in \mathcal{G}_j,
      \ndeg(G, H) \text{ is odd} \}
  \notag\\
  \bar{\mathcal{G}}_{j+1}^{O} &:= 
   \{G \in \mathcal{G}_{j+1}^{O}: G \text{ is a constant
                      vector field} \}
  \notag\\
  \mathcal{G}_{j+1}^E &:= \mathcal{G}_{j}^E \cup
  \{ \ad^{\ndeg(G,F)} G(H) :  
       G \in \bar{\mathcal{G}}_j^O, H \in \mathcal{G}_j,
      \ndeg(G, H) \text{ is even} \}
  \notag\\
  \mathcal{G}_{j+1} &:= \spn(\mathcal{G}_{j+1}^{O}) 
        + \mbox{cone}_{\geq 0}(\mathcal{G}_{j+1}^{E}).
                  \label{eq:Gjp1:brak}
\end{align}

The following summarizes results in \cite{HM_15i}; cf. 
\cite{Hor_67,SV_72ii,SV_72i, RB_06}.

\begin{Theorem}
  \label{thm:whor:stuck:var}
  Consider $\{X_t\}_{t \geq 0}$ solving \eqref{eqn:sdegen} under the assumption
  that $F$ is a polynomial, that $G_k$ is constant i.e.
  $X$-independent and suppose furthermore that the resulting dynamics
  is non-explosive as in \eqref{eq:no:blow}.  Assume that   \begin{align}
    \spn \biggl\{ H \in \bigcup_{j \geq 1}  \mathcal{G}_j^O 
            : H \text{ is a constant vector field} \biggr\}
      = \RR^n \,,
    \label{eq:spanning:cond}
  \end{align}
  then $\{X_t\}$ is a nice diffusion in the sense of \cref{def:nice:diff}.
\end{Theorem}

\begin{Remark}
  The condition \eqref{eq:spanning:cond} is special case of the
  H\"ormander (parabolic) sum of squares condition which asserts that if vector
  fields produced by the iterated Lie brackets
  \begin{align}
    G_1, \ldots, G_k, [G_1, F], \ldots, [G_k,F], [[G_1, F], F], 
    [[G_1,F],G_1] \ldots
  \end{align}
  span all of $\RR^n$ then the generator $\mathcal{L}$
  given by \eqref{fpe} along with $\mathcal{L}^*$,
  $\mathcal{L} \pm \partial_t$, $\mathcal{L}^* \pm \partial_t$ are all
  hypo-elliptic as in \cref{def:hypo:ellip}.  See
  \cite{hormander1967hypoelliptic} and more recently the treatment in
  \cite{RB_06}.
\end{Remark}

\section{Non-explosivity and Uniqueness Results}
\label{sec:uniqueness:Lor}

In this section we now return to the specific
setting~\eqref{eqn:lor63} and establish, subject to a non-degeneracy
condition on the noise, the hypo-ellipticity and irreducibility of
\eqref{eqn:lor63}.  Specifically we  establish that
\eqref{eqn:lor63} satisfies \cref{def:nice:diff} via
\cref{thm:whor:stuck:var} when $\gamma_1 > 0$ and at least one of
$\gamma_2, \gamma_3 $ is strictly positive.  

Let $\{X_t\}_{t \geq 0}$ denote
the process $(x_t, y_t, z_t)$ solving~\eqref{eqn:lor63}, and we will reuse the notations
$\tau_n, \tau, \LL, \mathcal{P}_t$, etc. from \cref{sec:methods} for $\{X_t\}_{t \geq 0}$.  In
particular note that \eqref{eqn:lor63} produces the infinitesimal
generator
\begin{align}
    \mathcal{L} 
       = \sigma(y-x) \partial_x+ [x(\rho-z)-y] \partial_y + [xy-\beta z] \partial_z
         +\gamma_1 \partial_x^2+ \gamma_2 \partial_y^2 + \gamma_3 \partial_z^2.
    \label{eq:Gen:L63:gen:case}
\end{align}

We now formulate the first result in this section.

\begin{Proposition}
\label{prop:noexp}
For any values $\sigma, \rho, \beta \in \RR$ and any
$\gamma_1, \gamma_2, \gamma_2 \geq 0$ the process $\{X_t\}_{t \geq 0}$
defined by \eqref{eqn:lor63} is nonexplosive in the sense of
\eqref{eq:no:blow}.  Moreover, if $\sigma > 0$ and either
$\gamma_1, \gamma_2>0$ or $\gamma_1, \gamma_3>0$, then
\eqref{eqn:lor63} is a nice diffusion in the sense of
\cref{def:nice:diff}.  Hence, in particular, the hypotheses of
\cref{prop:criteria} are satisfied for \eqref{eqn:lor63} if 
$\gamma_1, \gamma_2>0$ or $\gamma_1, \gamma_3>0$.
\end{Proposition}

\begin{proof}
  We first prove that $\{X_t\}_{t \geq 0}$ is non-explosive with the
  aid of \cref{prop:stab}.  Defining $H$ as in
  \eqref{eq:lyp:base:intro} we find that \eqref{eq:gen:lyp:b:intro}
  holds.  Thus taking $V = H$ we obtain \eqref{eq:LV:does:not:blow}
  from \eqref{eq:gen:lyp:b:intro} with Young's inequality, and the
  first assertion follows.

 To prove that $\{X_t\}_{t \geq 0}$ is a nice diffusion we
  proceed via \cref{thm:whor:stuck:var}.   Adopting
  the geometric notations as in \eqref{eq:Vec:diff:geo} we have
  \begin{align}
   F = \sigma(y-x) \partial_x+ [x(\rho-z)-y] \partial_y 
    + [xy-\beta z] \partial_z,
    \quad
  G_1 = \sqrt{2 \gamma_1} \partial_x,
    \quad
  G_2 = \sqrt{2 \gamma_2} \partial_y,
    \quad
  G_3 = \sqrt{2 \gamma_3} \partial_y.
  \end{align}
  Our task is now to exhibit a sequence of allowable Lie brackets
  between these fields to obtain the spanning condition
  \eqref{eq:spanning:cond}.  

  Start with the case $\gamma_1, \gamma_2 > 0$ and by viewing $G_1$ as
  the vector $( \sqrt{2 \gamma_1}, 0, 0)^T$, we have
  $F(\lambda G) = (-\lambda \sigma \sqrt{2 \gamma_1}, \lambda \rho
  \sqrt{2 \gamma_1}, 0)^T$, and therefore, cf. 
  \eqref{eq:rel:deg:cond}, $\ndeg(G_1, F)=1$.  Hence, by
  \eqref{eq:G1:brak}, we find that
\begin{align}
  G_1':=\text{ad}^1G_1(F)&=[G_1, F] 
     = -\sqrt{2 \gamma_1}\sigma \partial_x 
       + \sqrt{2 \gamma_1}(\rho-z) \partial_y 
       + \sqrt{2 \gamma_1}y \partial_z \in \mathcal{G}_1^O 
   \label{eq:first:brk}
\end{align}  
Next, from $\ndeg(G_2, G_1')=1$ and 
\eqref{eq:Gjp1:brak} follows
\begin{align}
  \tilde{G}_3:= \text{ad}^1 G_2(G_1') 
  = [G_2, G_1'] = \sqrt{2 \gamma_1 \gamma_2} \partial_z \in \mathcal{G}_2^O. 
\end{align}
Thus we have found $G_1, G_2, \tilde{G}_3 \in \bigcup_{j \geq 1} \mathcal{G}_j^O$
which together span $\RR^3$ and hence satisfy \eqref{eq:spanning:cond},
completing the proof in the first case.

Next, assume $\gamma_1, \gamma_3 > 0$.  As above again
$\ndeg(G_1, F)=1$ and \eqref{eq:first:brk} holds true.  On the other
hand, $\ndeg(G_3, G_1')=1$ and we compute
\begin{align}
  \tilde{G}_2 := \text{ad}^1 G_3(G_1') 
  = [G_3, G_1'] = -\sqrt{2 \gamma_1 \gamma_3} \partial_y \in \mathcal{G}_2^O. 
\end{align}
Here, we  found the spanning set
$G_1, \tilde{G}_2, G_3 \in \bigcup_{j \geq 1} \mathcal{G}_j^O$ satisfying
\eqref{eq:spanning:cond} as required by \cref{thm:whor:stuck:var}.
The proof is now complete.
\end{proof}

\section{Positive Recurrence in the Absence of Damping}
\label{sec:exist}

In this section, we study the dynamics~\eqref{eqn:lor63} in the case
when $\beta=0$ and $\gamma_1 > 0$.  Our goal is to show that
\eqref{eqn:lor63} has globally finite expected returns to some compact
set by constructing a Lyapunov function $V$ satisfying the condition
\eqref{eq:cru:Lyp:bound} in \cref{prop:stab} part (b).  In turn, this
result immediately implies the existence part of \cref{thm:main}~(ii)
as well as the uniqueness in the case when either
$\gamma_1, \gamma_2>0$ or $\gamma_1, \gamma_3 >0$ by way of
Proposition~\ref{prop:noexp}.

We state the main result of this section precisely as follows:
\begin{Proposition}
  \label{prop:pos:rec}
  Consider \eqref{eqn:lor63} in the case when $\sigma > 0$, $\beta = 0$,
  $\rho \in \RR$, $\gamma_1 > 0$ and $\gamma_2, \gamma_3 \geq 0$.  Then,
  there exists an $R > 0$  such that
  for any $S > 0$
  \begin{align}
    \sup_{|X| \leq S} \E_X \xi_{R} < \infty\,,
    \label{eq:uni:bound:comp:xi}
  \end{align}
  where $\xi_{R}$ is return time to the ball of radius $R$ as defined
  in \eqref{eq:return:ballz}.  Furthermore, when we make the
  additional assumption that either $\gamma_2 > 0$ or $\gamma_3 > 0$
  then \eqref{eqn:lor63} has a unique invariant probability measure.
\end{Proposition}

Regarding the organization of the section, \cref{sec:derivation:ly:1},
\ref{sec:derivation:ly:scl}, \ref{sec:glu} contain the derivation of a
Lyapunov function $V: \RR^3 \to \RR$ leading to
\eqref{eq:uni:bound:comp:xi} and the quantitative estimates implying
\eqref{eq:cru:Lyp:bound}.  The rigorous proof of \cref{prop:pos:rec} is
given in \cref{sec:main:thm:beta:z}.

\subsection{Derivation of the Lyapunov Function}
\label{sec:derivation:ly:1}

In order to simplify our analysis slightly in what follows we begin
with the preliminary observation that it is sufficient to address
special case when $\rho = 0$ in \eqref{eqn:lor63}, namely
\begin{align}
\nonumber dx&= \sigma (y-x) \, dt + \sqrt{2\gamma_1} \, dB_1,\\
          d y &= -xz \, dt  -y \,dt + \sqrt{2\gamma_2} \, dB_2,
\label{eqn:lor}\\
\nonumber dz&= xy \, dt + \sqrt{2\gamma_3} \, dB_3.
\end{align}        
Indeed, in the rest of this section we proceed to construct a function
$V\in C^2(\RR^3;$ $[0, \infty))$ such that for some constants $c,d>0$
and some compact set $\mathcal{K} \subseteq \RR^3$ we have
\begin{align}
\label{eqn:lyapexist}
\M V\leq - c + d\mathbf{1}_{\mathcal{K}}
\end{align}
where $\M$ is the infinitesimal generator of \eqref{eqn:lor}, that is, 
\begin{align}
  \M = \sigma(y-x) \partial_x -(xz+y) \partial_y 
       + xy \partial_z 
       + \gamma_1 \partial_x^2 
       + \gamma_2 \partial_y^2 
       + \gamma_3 \partial_z^2 .
  \label{eq:reduced:Lnz:gen}
\end{align}
Having found such  $V$ we obtain according to \cref{prop:stab},(b)
that
\begin{align}
  \E_X \tilde{\xi}_{\mathcal{K}} \leq \frac{V(X)}{c}, \quad
  \text{ where } 
  \tilde{\xi}_{\mathcal{K}} := 
  \inf\{ t\geq 0 \, : \tilde{X}_t \in \mathcal{K}\},
\end{align}
where $\tilde{X}_t = (\tilde{x}_t, \tilde{y}_t,\tilde{z}_t)$ obeys
\eqref{eqn:lor}.  Clearly
$X_t = (\tilde{x}_t, \tilde{y}_t, \tilde{z}_t - \rho)$ satisfies
\eqref{eqn:lor63} in the general case of any $\rho \in \RR$.  Thus,
if for any $R > \rho$ we denote
\begin{align}
  \xi_{R} := \inf\{ t \geq 0 : |X_t | \leq R\},
  \quad
  \tilde{\xi}_{R}  
     := \inf\{ t \geq 0 : | \tilde{X}_t | \leq R - \rho \},
\end{align}
then we have $\xi_{R} \leq \tilde{\xi}_{R}$.  Thus, by choosing $R > 0$
sufficiently large so that $\mathcal{K} \subset B_{R - \rho}$ we
obtain that
$\xi_{R} \leq \tilde{\xi}_{R} \leq \tilde{\xi}_{\mathcal{K}}$, so that
\begin{align}
  \E_X \xi_{R} \leq  \E_X \tilde{\xi}_{R} \leq \frac{V(X)}{c}
  \label{eq:wlog:smp:lor:concl}
\end{align}
allowing us to conclude \eqref{eq:uni:bound:comp:xi} as desired
in  \cref{prop:pos:rec}.

In order to find $V$ satisfying \eqref{eqn:lyapexist}, we first use
the natural Lyapunov function for \eqref{eqn:lor} when $\beta > 0$.
Indeed, defining
 \begin{align}
   \tilde{H}(x,y,z) = x^2 + y^2 + z^2 -2\sigma z + \kappa_0
 \end{align}
 where $\kappa_0 >0$ is large enough so that $\tilde{H} \geq 0$.
 Observe that $\tilde{H}$ provides a good initial guess for $V$ since
 \begin{align}
   \label{eqn:MH}
   \M (\tilde{H})
       = -2 \sigma x^2 -2y^2  
         + 2(\gamma_1 + \gamma_2 + \gamma_3)\,,
 \end{align}
 and therefore we have the desired
 inequality~\eqref{eqn:lyapexist} on the set where
 $|(x,y)|:=\sqrt{x^2+ y^2}$ is large.  More specifically, for the
 region
\begin{align}
  \mathcal{R}_0= \{ x^2 + y^2 \geq R_0\}
  \label{eq:region:R0}
\end{align}
with a sufficiently large  $R_0 \geq 1$ depending
only on $\gamma_1 + \gamma_2 + \gamma_3 > 0$, we have 
\begin{align}
  \M (\tilde{H}) \leq -(\gamma_1 + \gamma_2 + \gamma_3)
  \quad \text{ in } \mathcal{R}_0.
  \label{eq:good:Lyp:bnd:1}
\end{align}
However, \eqref{eq:good:Lyp:bnd:1} does not imply the 
bound \eqref{eqn:lyapexist} if $|(x,y)|$ is small (and $|z|$ is large).

To fix this issue, we seek for a lower-order perturbation $\psi$ of
$\tilde{H}$ encapsulating the `averaging' effects of the
dynamics. More specifically, we start with $\tilde{H}$ and find a
function $\psi\in C^2(\RR^3; \RR)$ satisfying
\begin{align}
  \limsup_{|X| \to \infty} \frac{\psi(X)}{\tilde{H}(X)} =  0 
  \label{eq:pos:Lyp}
\end{align}
so that $V:= \tilde{H} + \psi$ satisfies $V \geq 0$ and \eqref{eqn:lyapexist} for
some $c,d>0$ and the compact set 
\begin{align}
    \mathcal{K} := \{ x^2 + y^2 \leq R_0, |z| \leq R_3\}
  \label{eq:comp:exp:set}
\end{align}
for suitable choices of $R_0, R_3 \geq 1$.  

Note that, with this strategy, because $\tilde{H}$
satisfies~\eqref{eqn:MH} on $\mathcal{R}_0$, we should naturally
set $\psi=0$ on $\mathcal{R}_0$.  On the other hand, when
$x^2+y^2\leq R_0$ and $|z|$ is large, we should seek a nontrivial
perturbation $\psi$ through a scaling analysis to identify dominant
terms in $\M$.

\subsection{Scaling Arguments and Definition of $\psi$}
\label{sec:derivation:ly:scl}

To see how to define $\psi$ on the complement of $\mathcal{R}_0$, it
is helpful to first heuristically analyze the dynamics when $|z|$ is
large and $x$ and $y$ are bounded.  To this end, consider the scaling
transformation
 \begin{align}
 T_\lambda(x,y,z)=(\lambda^{-\alpha} x, y, \lambda z)
 \end{align}
 where $\lambda >1$ is large and $\alpha \in [0, 1]$.  We apply
 $T_\lambda$ to the generator $\M$ to formally see how the dynamics
 behaves as $z$ gets large.  Observe that
 \begin{align}
 \nonumber  
    T_\lambda \circ \M
    &=  \sigma( y \lambda^{\alpha} - x) \partial_x 
        -( \lambda^{1-\alpha} xz+ y) \partial_y 
        + \lambda^{-1-\alpha} xy \partial_z  
        +\gamma_1 \lambda^{2\alpha} \partial_x^2
        +\gamma_2 \partial_y^2  
        +\lambda^{-2}\gamma_3 \partial_z^2\\
 \label{eqn:scaling} 
    &\sim \lambda^{1-\alpha} x z \partial_y 
      + \gamma_1 \lambda^{2\alpha} \partial_x^2,
 \end{align}   
 whenever $\lambda \gg 1$ and $\alpha > 0$. 

 Observe that there are two regimes depending on $\alpha$.  If
  $\alpha \in [0, 1/3)$, the most
 significant term in \eqref{eqn:scaling} is
 $\lambda^{1-\alpha} xz \partial_y$.  Hence, the
 dominant dynamics of~\eqref{eqn:lor} is given by
\begin{align}
\label{eqn:lorapprox1}
\dot{X}=0, \qquad \dot{Y}= -XZ, \qquad \dot{Z}=0  
\end{align}  
and we expect such an approximation to be valid
in the region
\begin{align}
  \mathcal{R}_1:=\{x^2+ y^2 \leq R_0, 
      \,\, |x| |z|^{1/3} \geq R_1, |z| \geq R_3\}
  \label{eq:region:R1}
\end{align}
where $R_0, R_1, R_3 \geq 1$ are large constants to be determined
below. This suggests that we search for a function $\psi=\psi_1$ such
that the infinitesimal generator of \eqref{eq:region:R1} applied to
$\psi_1$ is negative
\begin{align}
  -xz \partial_y \psi_1 = - \kappa_1
\end{align}  
where $\kappa_1 >2(\gamma_1+ \gamma_2 + \gamma_3)$ is a constant.
Note that this equation gives the following particular solution
\begin{align}
\label{eqn:psi1}
\psi_1= \kappa_1 \frac{y}{xz}.
\end{align}
In addition,
 on the set $\mathcal{R}_1$ and positivity condition
\eqref{eq:pos:Lyp} holds and ($\M$ is the generator of \eqref{eqn:lor})
\begin{align} \label{mpb}
\frac{\M(\psi_1)}{\kappa_1}
  &= -\frac{\sigma(y - x)y}{x^2z} - 1 - \frac{y}{zx} - \frac{y^2}{z^2}
    + 2\gamma_1 \frac{y}{x^3z} + 2\gamma_3 \frac{y}{xz^3}\\
&\leq -1 + C\frac{R_0^3}{R_1}
\end{align}
where we used that on $\mathcal{R}_1$ one has
$|z|^{\frac{1}{3}} \geq \frac{R_1}{R_0}$ and $R_i \geq 1$ for each
$i = 1, 2, 3$. Note that the constant $C = C(\sigma, \gamma_3) > 0$ is
independent of $R_0,R_1,R_2$, and $\kappa_1$.  Thus, for sufficiently
large $R_1$ depending on $R_0$, we obtain
\begin{align}
\M(\psi_1) 
  &\leq - \frac{1}{2}\kappa_1 
    \quad \text{ in } \mathcal{R}_1,
    \label{eq:gd:bnd:psi1:R1}
\end{align}
Consequently,  for any fixed 
$R_0 \geq 1$, 
we can choose suitably large
$\kappa_1 \geq 1 \vee (4 (\gamma_1 + \gamma_2 + \gamma_3)$ and
$R_1 \geq 1$ so that
\begin{align}
  \mathcal{M}(\tilde{H} + \psi_1) \leq - \frac{\kappa_1}{2}
  \quad \text{ on the region } \mathcal{R}_1.
  \label{eq:good:Lyp:bnd:2}
\end{align}

Next, assume $\alpha \in (1/3,\infty)$ and observe that the dominant
term in \eqref{eqn:scaling} is
$\gamma_1 \lambda^{2\alpha} \partial_x^2$.  See
\cref{rmk:bnd:val:case} below which discusses the boundary case
$\alpha = 1/3$, where the two terms in \eqref{eqn:scaling} balance.
Therefore, the main contribution of the dynamics of \eqref{eqn:lor63}
in the region
\begin{align}
  \mathcal{R}_2= \{x^2+y^2 \leq R_2,  |x||z|^{1/3} \leq R_1, |z| \geq R_3 \},
 \label{eq:region:R2}
\end{align}
is given by the
SDE
\begin{align}
\label{eqn:approx2}
  dX = \sqrt{2\gamma_1} \, dB_1,  \quad
  \dot{Y} = 0,  \quad
  \dot{Z} = 0.  
\end{align} 
In the definition of $\mathcal{R}_2$, the constants $R_2$ and $R_3$
are considered sufficiently large, possibly depending on $R_0$.
\footnote{Note that additional parameter $R_2$ can be simply taken as
  $R_0$ in our preliminary analysis. However, it will play an
  important role later when we need to `glue' our Lypunov function $V$
  together to obtain a $C^2$ function.}  Thus, as above, in
$\mathcal{R}_2$, we should look for $\psi=\psi_2$ for which the
infinitesimal generator of \eqref{eqn:approx2} is negative, that is,
for $\psi_2$ that solves
\begin{align}
\gamma_1 \partial_x^2 \psi_2=-\kappa_2
\end{align}
where again $\kappa_2 \geq 1 \vee (4 (\gamma_1 + \gamma_2 + \gamma_3)$ is
a large free parameter we can adjust as suits our needs further on.
Note that a particular solution of this partial differential equation
is
\begin{align}
\label{eqn:psi2}
\psi_2= \frac{\kappa_2}{2\gamma_1}\bigg(\frac{4R_1^2}{|z|^{2/3}} - x^2\bigg).
\end{align} 
The solution is chosen such that it 
satisfies
\begin{align}\label{sep}
  |\psi_2| \leq C \frac{\kappa_2 R_1^2}{|z|^{2/3}}
  \qquad \textrm{whenever } |x| |z|^{1/3} \leq 2R_1.
\end{align}

As for the previous case, one can easily check \eqref{eq:pos:Lyp},
that is, $\psi_2$ is dominated by $\tilde{H}$ for large values of
$(x,y,z) \in \mathcal{R}_2$.  Moreover $\psi_2$ satisfies, for
$z\neq 0$,
\begin{align}
  \M (\psi_2)
  =-\kappa_2  - \frac{\kappa_2 \sigma}{\gamma_1} x(y-x) 
    - \frac{R_1^2 \kappa_2}{3\gamma_1} \frac{xy}{|z|^{2/3}} 
    + \frac{10 \kappa_2 R_1^2}{9}\frac{\gamma_3}{\gamma_1} \frac{1}{|z|^{8/3}}. 
\end{align}  
Also, in $\mathcal{R}_2$, using that $|x| \leq R_1/|z|^{1/3}$ and $R_i \geq 1$ for $i = 1, 2, 3$,
we have
\begin{align}
 \frac{\kappa_2 \sigma}{\gamma_1} |x(y-x)|
  \leq  \frac{2\kappa_2 \sigma}{\gamma_1} \frac{R_1 R_2^{1/2}}{|z|^{1/3}}
  \leq  \frac{2\kappa_2 \sigma}{\gamma_1} \frac{R_1^2 R_2}{R_3^{1/3}}.
\end{align}
Since the other terms (except $-\kappa_2$) have $|z|$ to some power in
the denominator, they are straightforward to estimate.  Overall, it
follows
\begin{align}
\M (\psi_2)
  &\leq -\kappa_2\left(1 - \frac{ C R_1^2 R_2}{R_3^{1/3}}\right),
    \label{eq:gd:bnd:psi2} 
\end{align}  
where the constant $C = C(\sigma, \gamma_1, \gamma_3)$ is independent of
$R_1,R_3, R_3$, and $\kappa_2$. Hence,  given
$R_2, R_1 \geq 1$, we choose large $R_3 \geq 1$ and 
$\kappa_2 \geq 1 \vee (4 (\gamma_1 + \gamma_2 + \gamma_3)$ so that
\begin{align}
  \M (\tilde{H} + \psi_2)  \leq - \frac{\kappa_2}{2}
  \quad \text{ in } \mathcal{R}_2
  \label{eq:good:Lyp:bnd:3}
\end{align}

Let us now make the preliminary definition
\begin{align}
  V := \tilde{H} + \indFn{\mathcal{R}_1} \psi_1 + \indFn{\mathcal{R}_2} \psi_2
  \label{eq:v:prelim}
\end{align}
and notice that the complement of compact region
$\mathcal{K} = \{ x^2 + y^2 \leq R_0, z \leq R_1\}$, as in
\eqref{eq:comp:exp:set} satisfies
\begin{align}
  \mathcal{K}^C \subseteq \mathcal{R}_0 \cup \mathcal{R}_1 \cup \mathcal{R}_2
  \label{eq:Komp:decomp}
\end{align}
Thus, setting aside the issue of 
differentiability of $V$, we can choose values
for $R_0, R_1, R_2, R_3 \geq 1$ and values for
$\kappa_1,\kappa_2 \geq 4(\gamma_1 + \gamma_2 + \gamma_3)$ such that
a combination of  \eqref{eq:good:Lyp:bnd:1}, \eqref{eq:good:Lyp:bnd:2} and
\eqref{eq:good:Lyp:bnd:3} leads to \eqref{eqn:lyapexist}.

The next section addresses the smoothness issue for $V$ defined as
\eqref{eq:v:prelim} by replacing indicator functions with smooth
cut-off functions.  We also provide estimates for additional terms
which are produced when the operator $\M$ acts on these smooth
cut-offs.

\begin{Remark}
  \label{rmk:bnd:val:case}
  One may be concerned that, when defining $\psi_1$ and $\psi_2$
  we neglected the effective dynamics of
  \eqref{eqn:lor63} in the critical region $\alpha=1/3$.  This
  is not a problem because the function $\psi_2$ is
  independent of $y$, and therefore it solves the associated PDE
  with both dominant terms
  \begin{align}
    -xz \partial_y \psi_2 
    + \gamma_1 \partial_x^2 \psi_2 
          = -\kappa_2.
  \end{align} 
\end{Remark}

\subsection{Gluing}
\label{sec:glu}

In order to replace the indicator functions in \eqref{eq:v:prelim}
with smooth cut-off functions we adopt the following
definitions. Consider $\chi$ and $\tilde{\chi}$ to be non-negative
$C^\infty(\RR)$ functions such that
\begin{align}
\chi(x)= 
\begin{cases}
1 & \text{ if } |x| \leq 1\\
0 & \text{ if } |x| \geq 2
\end{cases}
\quad 
\text{ and } 
\quad
\tilde{\chi}(x)= 
\begin{cases}
1 & \text{ if } |x| \geq 1\\
0 & \text{ if } |x| \leq 1/2.
\end{cases}. 
\end{align}
We now define\footnote{Observe that 
  for example $\tilde{\theta}_2^3(z)$ indicates that we are cutting off the
  region in $z$ (argument of the function) below the parameter value $R_3$ (tilde and superscript).}
\begin{align}
 \theta_1(x,y,z) := 
     \chi\bigg(\frac{x^2+y^2}{R_0}\bigg) 
     \tilde{\chi}\bigg(\frac{|x| |z|^{1/3}}{R_1 }\bigg) 
     \tilde{\chi}\bigg(\frac{|z|}{R_3}\bigg) 
     := \theta_1^0(x,y)\tilde{\theta}_1^1(x,z)\tilde{\theta}^3(z)
  \label{eq:Lyp:cut:1}
\end{align}
and put
\begin{align}
 \theta_2(x,y,z) := 
     \chi\bigg(\frac{x^2+y^2}{R_2}\bigg) 
     \chi\bigg(\frac{|x| |z|^{1/3}}{R_1 }\bigg)
     \tilde{\chi}\bigg(\frac{|z|}{R_3}\bigg) 
     := \theta_2^2(x,y)\theta_2^1(x,z)\tilde{\theta}^3(z).
  \label{eq:Lyp:cut:2}
\end{align}
We now define
\begin{align}
  V &:= \tilde{H} + \theta_1 \psi_1 + \theta_2 \psi_2
  \notag\\
    &= x^2 + y^2 + z^2 - 2 \sigma z + \kappa_0 + 
     \kappa_1 \theta_1(x,y,z) \frac{ y}{xz} + 
     \kappa_2 \theta_2(x,y,z) 
  \frac{1}{2 \gamma_1}\left( \frac{R_1^2}{|z|^{2/3}} - x^2\right).
  \label{eq:V:def}
\end{align}
Of course this definition requires the specification of the parameters
$R_0, R_1, R_2, R_3 \geq 1$ and $\kappa_0, \kappa_1, \kappa_2 > 0$, which
 will be clarified as we proceed with the argument. 

\subsection{Rigorous Bounds on $V$} 
\label{sec:main:thm:beta:z}
We are now ready to use $V$ defined in \eqref{eq:V:def} to prove the main result of this section.

\begin{proof}[Proof of \cref{prop:pos:rec}]
  As we identified in the argumentation leading to
  \eqref{eq:wlog:smp:lor:concl} above, it is sufficient to show that
  the $V$ defined by \eqref{eq:V:def} satisfies \eqref{eqn:lyapexist}
  and is strictly positive for suitable values of $R_0, R_1, R_2, R_3$
  and $\kappa_0, \kappa_1, \kappa_2$.  We emphasize that for the
  remainder of the proof, any constant $C > 0$ is
  independent of the values of the parameters $R_0, R_1, R_2, R_3$
  and $\kappa_0, \kappa_1, \kappa_2$ unless explicitly stated
  otherwise.

  Regarding the condition \eqref{eqn:lyapexist} we begin by observing
  that
\begin{align}
  \M(V) 
  =& \, \M(\tilde{H}) + \theta_1 \M(\psi_1) + \theta_2\M(\psi_2) 
  \notag\\
    &+\psi_1 \M (\theta_1)  +   2 \nabla_\gamma \theta_1 \cdot \nabla_\gamma \psi_1 +
    \psi_2\M(\theta_2) + 2 \nabla_\gamma \theta_2 \cdot \nabla_\gamma \psi_2,
      \label{eq:M:V:real:breakdown}
\end{align}
where we adopt the shorthand notation
$\nabla_\gamma = (\gamma_1 \partial_x, \gamma_2 \partial_y,
\gamma_3 \partial_z)$.
We proceed to expand each of the terms in
\eqref{eq:M:V:real:breakdown}, where derivatives fall on the cut-off
functions $\theta_1$ and $\theta_2$.  For later use we note
the estimate
\begin{gather}
  |\partial_s^i \theta_1^0| \leq C \indFn{x^2 + y^2 \leq 2R_0},
\end{gather}
where $s$ stands for $x$ or $y$ and $i \in \{1, 2\}$. Indeed, for example
\begin{equation}\label{gso}
  |\partial_y^2 \theta_1^0| \leq C \left( \frac{1}{R_0}
    + \frac{|y|^2}{R_0^2} \right) \indFn{x^2 + y^2 \leq 2R_0} \leq   C \indFn{x^2 + y^2 \leq 2R_0} 
\end{equation}
and other estimates follow analogously. In addition, we have
\begin{gather}
  |\partial_x   \tilde{\theta}_1^1| \leq C \frac{|z|^{1/3}}{R_1},
  \quad
  |\partial_x^2   \tilde{\theta}_1^1| \leq C \frac{|z|^{2/3}}{R_1^2},
  \quad
  |\partial_z   \tilde{\theta}_1^1|\leq C \frac{|x|}{|z|^{2/3} R_1},
  \quad
  |\partial_z^2   \tilde{\theta}_1^1| \leq 
  C \left( \frac{|x|}{|z|^{5/3} R_1}+ \frac{|x|^2}{|z|^{4/3} R_1^2}
    \right) ,
  \notag    \\
  |\partial_z  \tilde{\theta}^3|
         \leq C \frac{1}{R_3} \indFn{|z| \geq R_3/2},
  \quad
  |\partial_z^2  \tilde{\theta}^3| \leq C \frac{1}{R_3^2} \indFn{|z| \geq R_3/2}
  \label{eq:cut:est:th1}
\end{gather}
for a constant $C$ depending only and on the specifics of the cut-offs
$\chi$ and $\tilde{\chi}$ independent of $R_0, R_1$, and $R_3$.
Similarly
\begin{gather}
  |\partial_s^i \theta_2^2| \leq  C \indFn{x^2 + y^2 \leq 2R_2}  ,
\end{gather}
where $s$ stands for $x$ or $y$ and $i \in \{1, 2\}$ and
\begin{gather}
  |\partial_x   \theta_2^1| \leq C \frac{|z|^{1/3}}{R_1},
  \quad
  |\partial_x^2   \theta_2^1| \leq C \frac{|z|^{2/3}}{R_1^2},
  \quad
  |\partial_z   \theta_2^1|\leq C \frac{|x|}{|z|^{2/3} R_1},
  \quad
  |\partial_z^2   \theta_2^1| \leq 
  C \left( \frac{|x|}{|z|^{5/3} R_1}+ \frac{|x|^2}{|z|^{4/3} R_1^2}
    \right) ,
  \label{eq:cut:est:th2}
\end{gather}
where again $C >0$ is independent of $R_1, R_2$, and $R_3$.  Observe
that $\tilde{\theta}^3$ is the same in both $\theta_1$ and $\theta_2$.
Denote $K_{R_3}$ a constant that might depend on $R_0$, $R_1$, and
$R_2$ such that
\begin{align}
\label{eqn:constks}
\lim_{R_3 \to \infty} K_{R_3} = 0 \,.
\end{align}

We expand $\psi_1 \M (\theta_1)$ as 
\begin{align}
  \psi_1 \M (\theta_1) =& \, 
  \sigma(y-x) \psi_1 
  (\partial_x\theta_1^0\tilde{\theta}_1^1\tilde{\theta}^3 +  
   \theta_1^0 \partial_x\tilde{\theta}_1^1\tilde{\theta}^3)
  -(xz+y)\psi_1  
      \partial_y \theta_1^0\tilde{\theta}_1^1\tilde{\theta}^3
   + xy \psi_1 
  ( \theta_1^0\partial_z\tilde{\theta}_1^1\tilde{\theta}^3
    + \theta_1^0\tilde{\theta}_1^1 \partial_z\tilde{\theta}^3)
                          \notag\\
  &+ \gamma_1 \psi_1 
      (\partial_x^2 \theta_1^0\tilde{\theta}_1^1\tilde{\theta}^3
       +  \theta_1^0 \partial_x^2\tilde{\theta}_1^1\tilde{\theta}^3
       + 2\partial_x  \theta_1^0 \partial_x\tilde{\theta}_1^1\tilde{\theta}^3)
  + \gamma_2 \psi_1 \partial_y^2 \theta_1^0\tilde{\theta}_1^1\tilde{\theta}^3
                        \notag\\
  &+ \gamma_3 \psi_1 
    (\theta_1^0\partial_z^2\tilde{\theta}_1^1\tilde{\theta}^3
    +\theta_1^0 \tilde{\theta}_1^1 \partial_z^2\tilde{\theta}^3
    +2\theta_1^0 \partial_z \tilde{\theta}_1^1 \partial_z \tilde{\theta}^3) \,.
    \label{eq:cutoff:vomit:1}
\end{align}
Since on $\mathcal{R}_1$  one has 
\begin{equation}\label{epo}
|\psi_1| = \kappa_1 \left|\frac{y}{xz}\right| \leq \frac{\kappa_1 R_0^{\frac{1}{2}}}{R_1} \frac{1}{|z|^{\frac{2}{3}}}
\end{equation}
and $x, y$ are bounded, it is easy to check that all terms except
$xz \psi_1 \partial_y \theta_1^0\tilde{\theta}_1^1\tilde{\theta}_1^3$
and
$\gamma_1 \psi_1 \theta_1^0
\partial_x^2\tilde{\theta}_1^1\tilde{\theta}_1^3$ can be bounded by
$\kappa_1 K_{R_3}$ (some power of $z$ is left in the denominator).
Referring back to \eqref{gso}, \eqref{eq:cut:est:th1}, and \eqref{epo}
we have
\begin{align}
  | xz \psi_1 \partial_y \theta_1^0\tilde{\theta}_1^1\tilde{\theta}^3|
  &= 
 \kappa_1 |y \partial_y \theta_1^0\tilde{\theta}_1^1\tilde{\theta}^3|
    \leq C \kappa_1 \indFn{R_0 \leq x^2 + y^2 \leq 2R_0}
    \leq  C \kappa_1 \indFn{\mathcal{R}_0}\\
  |\gamma_1 \psi_1 \theta_1^0 \partial_x^2\tilde{\theta}_1^1\tilde{\theta}^3|
  &\leq C  \frac{\kappa_1 R_0^{\frac{1}{2}}}{R_1} \frac{1}{|z|^{\frac{2}{3}}}  \frac{|z|^{\frac{2}{3}}}{R_0^2}
  \leq C \kappa_1 \frac{1}{R_1} \,,
  \label{eq:bad:ass:tm:1}
\end{align}
where the constant $C$ is independent of $R_0, R_1, R_2, R_3$ and
$\kappa_0, \kappa_1, \kappa_2$.  Overall, we have
\begin{equation}\label{opm}
  |\psi_1 \M (\theta_1)|
  \leq C\kappa_1 \left(\indFn{\mathcal{R}_0} + \frac{1}{R_1} + K_{R_3} \right) \,
\end{equation}
where we recall that $K_{R_3}$ is as in~\eqref{eqn:constks}.

Next, we estimate
\begin{align}         
  |\nabla_\gamma& \theta_1 \cdot \nabla_\gamma \psi_1| \notag\\
    &\leq C \kappa_1
      \left(
        \frac{|y|}{|x|^2|z|} 
          |\partial_x\theta_1^0\tilde{\theta}_1^1\tilde{\theta}^3 +  
              \theta_1^0 \partial_x\tilde{\theta}_1^1\tilde{\theta}^3|
          + 
        \frac{1}{|x||z|} 
          |\partial_y \theta_1^0\tilde{\theta}_1^1\tilde{\theta}^3|
          +
        \frac{|y|}{|x||z|^2} 
         | \theta_1^0\partial_z\tilde{\theta}_1^1\tilde{\theta}^3
         + \theta_1^0\tilde{\theta}_1^1 \partial_z\tilde{\theta}^3 | 
      \right)
          \notag\\
    &\leq C \kappa_1
      \left(
        \frac{R_0^{\frac{1}{2}}}{R_1^2 |z|^{\frac{1}{3}}} 
          |\partial_x\theta_1^0\tilde{\theta}_1^1\tilde{\theta}^3 +  
              \theta_1^0 \partial_x\tilde{\theta}_1^1\tilde{\theta}^3|
          + 
        \frac{1}{R_1 |z|^{\frac{2}{3}}} 
          |\partial_y \theta_1^0\tilde{\theta}_1^1\tilde{\theta}^3|
          +
        \frac{R_0^{\frac{1}{2}}}{R_1 |z|^{\frac{5}{3}}} 
         | \theta_1^0\partial_z\tilde{\theta}_1^1\tilde{\theta}^3
         + \theta_1^0\tilde{\theta}_1^1 \partial_z\tilde{\theta}^3 | 
      \right)
         \notag\\      
   &\leq C \kappa_1
      \left(  \frac{R_0^{\frac{1}{2}}}{R_1^4} +  
      K_{R_3}         
      \right) \,.\label{eq:grad:gam:tm:1}
\end{align}

We next estimate the cut-off terms involving $\psi_2$.  Similar to
\eqref{eq:cutoff:vomit:1}, we can write $\psi_2 \M (\theta_2)$ as
\begin{align}
  \psi_2 \M (\theta_2) =& \, 
  \sigma(y-x) \psi_2 
  (\partial_x\theta_2^2\theta_2^1\tilde{\theta}^3 +  
   \theta_2^2\partial_x\theta_2^1\tilde{\theta}^3)
  -(xz+y)\psi_2  
      \partial_y\theta_2^2\theta_2^1\tilde{\theta}^3
   + xy \psi_1 
  ( \theta_2^2 \partial_z \theta_2^1\tilde{\theta}^3 +  
    \theta_2^2\theta_2^1 \partial_z \tilde{\theta}^3 )
                         \\
  &+ \gamma_1 \psi_2 
      (\partial_x^2 \theta_2^2\theta_2^1\tilde{\theta}^3+
    \theta_2^2 \partial_x^2 \theta_2^1\tilde{\theta}^3 +
    2 \partial_x \theta_2^2 \partial_x \theta_2^1\tilde{\theta}^3)
  + \gamma_2 \psi_2 
    \partial_y^2 \theta_2^2\theta_2^1\tilde{\theta}^3+
                    \\
  &+ \gamma_3 \psi_2 
    (\theta_2^2\partial_z^2\theta_2^1\tilde{\theta}^3+
    \theta_2^2\theta_2^1\partial_z^2\tilde{\theta}^3+
    2\theta_2\partial_z\theta_2^1\partial_z\tilde{\theta}^3). \label{eq:wev}
\end{align}
Due to the presence of $\theta_2^1$ and/or its derivatives, each term
in \eqref{eq:wev} is supported on the set
$\{|x||z|^{\frac{1}{3}} \leq 2 R_1\}$, and therefore the estimate
\eqref{sep} applies. Similar to the above, the only terms that cannot
be estimated by $K_{R_3}$ are
$xz\psi_2 \partial_y\theta_2^2\theta_2^1\tilde{\theta}^3 $ and
$\gamma_1 \psi_2 \theta_2^2 \partial_x^2 \theta_2^1\tilde{\theta}^3$,
and for those we have
\begin{align}
  |xz\psi_2 \partial_y\theta_2^2\theta_2^1\tilde{\theta}^3| 
  \leq C \kappa_2 \frac{|x||z|R_1^2|y|}{|z|^{2/3}R_2}
  \leq C \kappa_2 \frac{R_1^3}{R_2^{1/2}}
   \label{eq:bad:ass:tm:2:2}
\end{align}
and, by definition of $\theta_1$,
\begin{align}
  | \gamma_1 \psi_2 \theta_2^2 \partial_x^2 \theta_2^1\tilde{\theta}^3|
  \leq& C \kappa_2 \tilde{\theta}^3 
      \indFn{ x^2 + y^2 \leq 2R_2, R_1 \leq |x||z|^{1/3} \leq 2R_1}
  \notag\\
  \leq& C \kappa_2 \tilde{\theta}^3 
      \indFn{ x^2 + y^2 \leq R_0, R_1 \leq |x||z|^{1/3} \leq 2R_1}+
  C \kappa_2 \indFn{ \mathcal{R}_0}
\notag\\
  \leq& C \kappa_2 \theta_1 +
  C \kappa_2 \indFn{ \mathcal{R}_0}  
  \,.
  \label{eq:bad:ass:tm:2:3}
\end{align}
Hence, we have 
\begin{align}
  |\psi_2 \M(\theta_2)| \leq C\kappa_2 \left( \frac{R_1^3}{R_2^{1/2}} +  \theta_1 +
  \indFn{ \mathcal{R}_0} + K_{R_3}\right) 
  \label{eq:good:tm:N:2}
\end{align}
where $K_{R_3}$ is as in~\eqref{eqn:constks}.

After expanding $\nabla_\gamma \theta_2 \cdot \nabla_\gamma \psi_2$,
the only terms that cannot be bounded by $K_{R_3}$ are
$\gamma_1 \partial_x \psi_2 \partial_x
\theta_2^2\theta_2^1\tilde{\theta}_2^3$ and
$\gamma_1 \partial_x \psi_2
\theta_2^2\partial_x\theta_2^1\tilde{\theta}_2^3$.  However, if
$R_2 \geq R_0$
\begin{align}
|\gamma_1 \partial_x \psi_2
\partial_x \theta_2^2\theta_2^1\tilde{\theta}^3| 
 \leq C \kappa_2 \frac{|xy|}{R_2}  
      \indFn{ x^2 + y^2 \leq 2R_2}
  \leq C \kappa_2  \indFn{R_2 \leq x^2 + y^2 \leq 2R_2} \leq C \kappa_2 \indFn{ \mathcal{R}_0} 
\end{align}
and on $\mathcal{R}_2$
\begin{align}
|\gamma_1 \partial_x \psi_2
 \theta_2^2\partial_x\theta_2^1\tilde{\theta}^3| 
  \leq& C \kappa_2  \frac{|xy||z|^{\frac{1}{3}}}{R_2R_1}
        \indFn{ x^2 + y^2 \leq 2R_2, |x||z|^{\frac{1}{3}} \leq 2R_1}
        \leq 
C \frac{\kappa_2}{R_2^{\frac{1}{2}}}  \,.
\end{align}
Overall,
\begin{align}
  |\nabla_\gamma \theta_2 \cdot \nabla_\gamma \psi_2|
  &\leq C\kappa_2\left(   \indFn{\mathcal{R}_0} + \frac{1}{R_2^{\frac{1}{2}}}
    + K_{R_3}\right) \,.
         \label{eq:grad:gam:tm:2}
\end{align}
Let us now gather the estimates \eqref{eqn:MH}, \eqref{mpb},
(\ref{eq:gd:bnd:psi2}), \eqref{opm}, \eqref{eq:grad:gam:tm:1},
\eqref{eq:good:tm:N:2}, and \eqref{eq:grad:gam:tm:2} to obtain for
$R_2 \geq R_0$
\begin{align}
  \M(V) \leq & 
      -2 \sigma x^2 -2y^2  + \bar{\gamma}
      -\kappa_1 \theta_1 \left(1 -   \frac{CR_0^{3}}{R_1} \right)
      -\kappa_2 \theta_2 \left(1 - \frac{ C R_1^2 R_2}{R_3^{1/3}}\right)
  \notag\\
     &+ C (\kappa_1 + \kappa_2) \indFn{\mathcal{R}_0} + C\kappa_1\left(\frac{1}{R_1} +      
     \frac{R_0^{\frac{1}{2}}}{R_1^4} \right) + C\kappa_2 \frac{R_1^3}{R_2^{1/2}}  
      + C \kappa_2 \theta_1 
 + K_{R_3}(\kappa_1 + \kappa_2)    \,,        
\end{align}
where $\bar{\gamma} := 2(\gamma_1+ \gamma_2 + \gamma_3)$.  Let us fix
$\kappa_2 = 16 \bar{\gamma}$, $\kappa_1$ such that
$\frac{\kappa_1}{2} \geq \max\{4\bar{\gamma}, C \kappa_2 \}$ and
$R_0 \geq 1$ such that
\begin{align}
  (2\sigma x^2 + y^2) \geq 4 \bar{\gamma} + C(\kappa_1 + \kappa_2)
  \qquad \textrm{ in } \mathcal{R}_0 \,.
\end{align}
Then, choose $R_1$ such that 
\begin{equation}
C\kappa_1\left(\frac{1}{R_1} +      
  \frac{R_0^{\frac{1}{2}}}{R_1^4} \right)
\leq \frac{\bar{\gamma}}{3}
\qquad \textrm{and }  \quad \frac{CR_0^{3}}{R_1}  \leq \frac{1}{2} 
\end{equation}
and $R_2 \geq R_0$ such that 
\begin{equation}
 C\kappa_2 \frac{R_1^3}{R_2^{1/2}} \leq \frac{\bar{\gamma}}{3} \,.
\end{equation}
Finally, choose $R_3$ such that 
\begin{equation}
K_{R_3}(\kappa_1 + \kappa_2) \leq \frac{\bar{\gamma}}{3} \qquad \textrm{ and }\quad 
\frac{ C R_1^2 R_2}{R_3^{1/3}} \leq \frac{1}{4} \,.
\end{equation}

With these parameter selections and referring back to
\eqref{eq:Lyp:cut:1}, \eqref{eq:Lyp:cut:2} we therefore have
\begin{align}
  \M(V) &\leq -4 \bar{\gamma} \indFn{\mathcal{R}_0} + 2\bar{\gamma} 
  - \frac{\kappa_1}{2} \indFn{\mathcal{R}_1} - \frac{\kappa_2}{4}\indFn{\mathcal{R}_2}
              \leq -2\bar{\gamma}   + 
     4\bar{\gamma} (1 - \indFn{\mathcal{R}_0 \cup \mathcal{R}_1 \cup \mathcal{R}_2})
\end{align}
Since $R_1 \leq R_2$, one has
$\{x^2 + y^2 \leq R_0, |z| \geq R_3\} \subset \mathcal{R}_1 \cup
\mathcal{R}_2$, and therefore
$(1 - \indFn{\mathcal{R}_0 \cup \mathcal{R}_1 \cup \mathcal{R}_2}) =
\indFn{\mathcal{K}}$, where
$\mathcal{K} \subset \{x^2 + y^2 \leq R_0, |z| \leq R_3\} $ is
bounded. Consequently, \eqref{eqn:lyapexist} follows with
$c = 2\bar{\gamma}$ and $d = 4\bar{\gamma}$.

Finally let us address the non-negativity of $V$.  Notice that our
selection of the parameters $R_0$, $R_1$, $R_2$, $R_3$ and of
$\kappa_1,\kappa_2$ was made independent of the value $\kappa_0$ (see \eqref{eq:V:def}).
Notice however that, by \eqref{epo} we have
\begin{align}
  |\theta_1 \psi_1| \leq C \kappa_1 \frac{R_1^{1/2}}{R_1R_3^{2/3}}.  
\end{align}
Similar to~\eqref{sep}
we observe that 
\begin{align}
  |\theta_2 \psi_2| \leq C \kappa_2 \frac{R_1^2}{R_3^{2/3}}.
\end{align}
Thus having fixed $R_0, R_1, R_2,R_3, \kappa_1, \kappa_2$ and referring
back to \eqref{eq:V:def} we have
\begin{align}
  V \geq x^2 + y^2 + z^2  - \sigma^2 - C \kappa_1 \frac{R_1^{1/2}}{R_1R_3^{2/3}}
  C \kappa_2 \frac{R_1^2}{R_3^{2/3}} + \kappa_0
\end{align}
making clear that $\kappa_0$ can be selected so that $V$ is positive for
every $(x,y,z) \in \RR^3$.  The proof is now complete.
\end{proof}

\section{Sensitivity with Respect to Convective Forcing.}
\label{sec:sensitive:noise}

This section addresses some special cases of a very degenerate
stochastic forcing when $\beta = 0$ in \eqref{eqn:lor63}.  First, we
establish Theorem \ref{thm:main} part (ii) by using the
test function $M$ given in \eqref{eq:z:cru:test}.

\subsection{Non-existence under highly degenerate noise} 
Before proceeding to the rigorous proof of Theorem \ref{thm:main},
(ii) we present a formal argument.  Suppose that we had an invariant
probability measure $\mu$ for \eqref{eqn:lor63} with
$\beta = \gamma_1 =0$.  Let us proceed with the unjustified assumption
that
\begin{align}
  \int_{\RR^3} |X|^2 \mu(dX) < \infty.
  \label{eq:fin:mom:formal}
\end{align}
Applying It\^o's formula to the function
$M(x,y,z) := 2\sigma z - x^2$, with the process $(x_t, y_t, z_t)$
initially distributed according to such an invariant measure $\mu$, we
obtain
\begin{align}
  \E_\mu [2\sigma z_t - x^2_t] = \E_\mu [2\sigma z_0 - x^2_0] 
          + \E_\mu \int_0^t [2\sigma xy - 2x\sigma(y- x)] ds \,.
\end{align}
Thus, stationarity and simple algebraic manipulations,
cf. \eqref{eq:gen:M:test}, imply that
\begin{align}
\E_\mu \int_0^t  x^2 \, ds= 0  \,,
\end{align}
so that $x_t \equiv 0$ for every $t \geq 0$.  

Now we address two cases.  First, we suppose that that $\gamma_3 > 0$.
In this situation we apply It\^o's lemma to $z^2$ and use that
$x_t \equiv 0$ to find
$dz^2 = 2 \gamma_3 dt + 2\sqrt{2 \gamma_3}zdB_2$. Integrating and
taking expectations we obtain
\begin{align}
  \E_\mu z_t^2 = \E_\mu z_0^2 + 2 \gamma_3
\end{align}
which contradicts stationarity if $\gamma_3 > 0$.  

Now let us consider the second case when $\gamma_3 = 0$ but
$\gamma_2 > 0$.  In this situation the stationary process
$\tilde{X} := (x_t,y_t)$ started with initial conditions distributed
according to the first two components of the invariant probability
measure $\mu$ which satisfies the first two components of
\eqref{eqn:lor63} maintains
\begin{align}
  d x = \sigma y dt, \quad 
  x_0 = 0,  
  \qquad
  d y =  - y dt + \sqrt{2 \gamma_2} dB_2\,.
  \label{eq:red:eq:gm2}
\end{align}
Here, once again we are using that $x_t \equiv 0$ we  obtain 
\begin{equation}
\sigma y = \frac{dx}{dt} = 0 \,,
\end{equation}
and therefore $y = 0$ a contradiction to $\gamma_2 \neq 0$. 

To make the above arguments rigorous and avoid the assumption 
\eqref{eq:fin:mom:formal}, we use cut-off functions and
carefully pass to a limit.  We now provide the details.  

\begin{proof}[Proof of \cref{thm:main}, (ii)]
Let $h : [0, 2] \to \RR$ be a non-decreasing $C^2$ function such that 
\begin{equation}
h(0) = h''(0) = h'(2) = h''(2) = 0, \qquad h'(0) = 1, \qquad h(2) = 1
\end{equation}
and $\max_{[0, 2]}|h'| \leq 1$. Denote $c^* = \max_{[0, 2]}|h''|$. It
is easy to see that such a function indeed exists. For each $N \geq 1$,
define a $C^2$ function $F_N: \RR \to \RR$ as an odd function with
\begin{equation} \label{fnd}
F_N(x) = \begin{cases}
x & x \in [0, N] \,, \\
h(x - N) + N & x \in [N, N + 2] \,,\\
N + 1 & x \geq N + 2 \,.
\end{cases}
\end{equation}
Note that $F_N' \geq 0$, $\max_{[0, 2]}|F_N'| \leq 1$, and
$\max_{[0, 2]}|F_N''| = c^*$.

To obtain a contradiction, assume that there is an invariant probability
measure $\mu$ of \eqref{eqn:lor63} and let $(x, y, z)$ have law
$\mu$. Since $\mu$ is a probability measure, there exists an
increasing sequence of integers $(N_j)_{j = 1}^\infty$ with
$N_{j + 1} - N_j \geq 2$ such that
\begin{equation}\label{pto}
\lim_{j \to \infty} \Prb(|2\sigma z - x^2| \in [N_j, N_j + 2]) = 0 \,.
\end{equation}
If we apply It\^ o's formula to $F_N(2\sigma z - x^2)$, we obtain
\begin{align}
\E_\mu F_N(2\sigma z_t - x^2_t) =& \E_\mu F_N(2\sigma z_0 - x^2_0)) \\
&+ \E_\mu \int_0^t (F_N'(2\sigma z - x^2) (2\sigma xy - 2x \sigma (y - x))  
                  + F_N''(2\sigma z - x^2) 4\sigma^2 \gamma_3) ds \,.
\end{align}
Simple algebraic manipulations and stationarity yield
\begin{equation} 
 \E_\mu x^2  F_N'(2\sigma z_t - x^2_t)  
     = -2\sigma \gamma_3 \E_\mu F_N''(2\sigma z_t - x^2_t)    \,. 
\label{ide}
\end{equation}
Next, we verify that $F'_{N_{j+1}} \geq F'_{N_{j}}$ for any
$j$. Indeed, for $|\xi| \leq N_j$ one has
$1 = F'_{N_{j}}(\xi) = F'_{N_{j+1}}(\xi)$ and for
$|\xi| \geq N_{j + 2}$ one has
$ F'_{N_{j}}(\xi) = 0 \leq F'_{N_{j+1}}(\xi)$. Finally, since
$N_{j + 1} \geq N_j + 2$, for any $|\xi| \in [N_j, N_j + 2]$, we have
$F'_{N_{j}}(\xi) \leq 1 = F'_{N_{j+1}}(\xi)$.  Thus, $(F_{N_j}')$ is
an non-decreasing sequence of non-negative functions that converge
pointwise to 1 on $\RR$. Therefore, by the monotone convergence
theorem and \eqref{ide}, we have
\begin{equation}\label{eox}
  \E x^2 = \lim_{j \to \infty} \E x^2  F_{N_j}'(2\sigma z - x^2)
  = -2\sigma \gamma_3 \lim_{j \to \infty}   \E F_{N_j}''(2\sigma z - x^2) \,.
\end{equation} 
Finally, from $|F_{N}''| \leq c^*$, $F_N'' = 0$ on the complement of
$[N, N + 2]$, and \eqref{pto} follows
\begin{equation}\label{seo}
  \lim_{j \to \infty}  \E F_{N_j}''(2\sigma z - x^2)
  \leq c^* \Prb(2\sigma z - x^2 \in [N_j, N_j + 2]) = 0 \,.
\end{equation}
Combining \eqref{eox} and \eqref{seo} yields $\E x^2 = 0$. However, if
$\E x^2 = 0$, then, $x = 0$ almost surely and by the third equation of
the Lorenz system, we have $z(t) = z(0) + \sqrt{2\gamma_3} B_3(t)$.
This contradicts invariance.
\end{proof}

\subsection{Uniqueness when the noise component acts only on the
  convection component of the system}

We next turn to the case when $\gamma_1 > 0$ but
$\beta = \gamma_2 = \gamma_3 = 0$.  In this special case 
of \cref{thm:main} (i), we can moreover give an explicit 
form for the invariant probability measure.

\begin{Proposition}\label{lem:mnx}
  Consider \eqref{eqn:lor63} with $\sigma > 0$ and $\rho \in \RR$.  If
  $\gamma_1 > 0$, $\gamma_2 = \gamma_3 = 0$, and $\beta = 0$, then
  \eqref{eqn:lor} has precisely one statistically invariant state
given by the product measure
  \begin{align}
   \mu = \nu_{0, \gamma_1/\sigma} \times \delta_0 \times \delta_\rho  
    \label{eq:stat:sol:degen}
  \end{align}
  where $\delta_a$ is the Dirac measure concentrated at $a$ and $\nu_{m,s}$
  is the 1-d Guassian measure with mean $m$ and variance $s$
\end{Proposition}

Once again, before proceeding to a rigorous proof, we present a
formal argument.  Suppose that in this parameter range there exists
an invariant probability measure $\mu$ of \eqref{eqn:lor63} and impose
the apriori unjustified condition \eqref{eq:fin:mom:formal}.  Let
$(x,y,z)$  be the solution starting for an initial condition
distributed as $\mu$. Observe that
\begin{align}
  \frac{1}{2} \frac{d}{dt}(y^2 + (z - \rho)^2) 
  = x (\rho - z) y - y^2 + xy(z - \rho)
  =  - y^2.
\end{align}
Integrating this expression in time, taking expected values and using
stationarity one finds that $\E_\mu\int_0^t y^2 ds = 0$ so that
$y_t \equiv 0$ for every $t \geq 0$ by path continuity.  Then, as a
consequence of this calculation, we infer that $\frac{d}{dt}z = 0$ so
that $z_t \equiv z_0$ for every $t \geq 0$.  Thus, with stationarity,
the equation for $\tilde{X}_t = (x_t, y_t)$ reduces to

\begin{align}
  d x = - \sigma x dt + \sqrt{2 \gamma_1} dB_1, \quad
  \qquad 
  d y = (z_0 - \rho) x dt,
  \quad
  y_0 = 0.  
\end{align}
The stationarity implies that for every $t, T$
\begin{equation}
(z_0 - \rho)  \E \int_t^T x ds = \E y(T) - \E y(t) = 0 \,.
\end{equation}
Since $x$ is almost surely continuous, either $z_0 = \rho$ or $x =
0$. The latter case leads to a immediate contradiction, whereas the
former one implies that \eqref{eq:stat:sol:degen} is the only
invariant state of (\ref{eqn:lor63}).

\begin{proof}[Proof of Proposition~\ref{lem:mnx}]
  By Theorem~\ref{thm:main} (i), there exists an invariant probability
  measure $\mu$, and let $(x, y, z)$ be a random initial condition
  distributed according to $\mu$.  For each $N \geq 1$, let $F_N$ be
  as in \eqref{fnd}.  Similar to the above, fix an increasing sequence
  $(N_j)_{j = 0}^\infty$ such that $N_{j + 1} \geq N_j$.  Then,
  applying It\^o's formula to $F_N(y^2 + z^2)$ and taking expected
  values gives
\begin{align}
  \E F_N(y^2_t + z^2_t) = \E F_N(y^2_0 + z^2_0) 
  + \E \int_0^t F'_N(y^2 + z^2)(-2y(xz + y) + 2zxy) ds \,.
\end{align}
Since the process is stationary, we have
\begin{equation}
\E F'_N(y^2 + z^2) y^2 = 0 \,.
\end{equation}
As in the proof of  \cref{thm:main} (ii), by using that $(F_{N_j}')$ is an
increasing sequence converging pointwise to 1, the monotone
convergence theorem implies
\begin{equation}
\E y^2 = 0 \,.
\end{equation}
However, if $y = 0$ almost surely, then $z' = 0$, and therefore
$z_t = z_0$, and $x$ is an invariant state of
\begin{equation}
dx = -\sigma x dt + \sqrt{2\gamma_1} dB_1
\end{equation}
as desired. 
\end{proof}

\section{Non-Existence of Stationary States
  in the Presence of a Linear Instability}
\label{sec:non-existence}

In this section, we prove Theorem~\ref{thm:main} part (iii) by
constructing functions $V_i$ satisfying the hypotheses of
Theorem~\ref{thm:criterianon}.  In the expressions that follow, we
assume that all constants depend implicitly on
$\sigma, \beta, \gamma_1, \gamma_2$, and $\gamma_3$.  Any other
dependence will be indicated explicitly.

\subsection{Construction overview}
Before proceeding to the proof, let us overview the construction of
$V_1$ and $V_2$ needed to apply Theorem \ref{thm:criterianon}.  We
remark that the function $V_1$ identifies `bad' initial conditions
from which the dynamics takes too long to return near the origin.
Because $\beta <0$, we note that the $z$ process in
equation~\eqref{eqn:lor63} grows exponentially fast when it is
initially large and when the product $xy$ is not too large.  In fact,
if one considers the test function
\begin{equation}
M(x, y, z) = 2\sigma z - x^2 \,, 
\end{equation}
then we note that 
\begin{align}
\LL M(x,y,z) = 2\sigma( |\beta | z+ x^2) - 2\gamma_1.
\end{align}
Hence, if $x^2$ is dominated by $z$, then the system~\eqref{eqn:lor63}
grows exponentially fast on average.  However, we have to be careful
because the noise can drive the dynamics out of the region
$\{x^2 < |\beta|z\}$.  To see that such scenario does not occur with
high enough probability, we have to modify $M$ and choose appropriate
$V_2$.

Let us first discuss possible candidates for $V_2$.  It is easy to
check that $\mathcal{L} H$ is neither bounded from above nor from
below, and therefore it is not a suitable choice for $V_2$. However,
we will see that $\mathcal{L} (\ln H)$ is bounded outside of a compact
set, and as such we use an appropriate multiple of $\ln H$ for the
function $V_2$.  To satisfy the assumption (p3) in Theorem
\ref{thm:criterianon}, it is necessary that $V_1$ has smaller than
logarithmic increase at infinity. Given the analysis above, a natural
choice would be $F\circ M$, with slowly growing $F$.  However, unlike
$H$, $M$ does not have a definite sign, and therefore to define
$V_1 = F \circ M$ one has to define $F$ on the whole real line. We
will verify below that $F(\zeta) = \ln \ln \zeta$ indeed produces
$\mathcal{L}(F\circ M(x, y, z)) \geq 0$ for large $M(x, y, z)$, but
$F$ is not even defined for $M(x, y, z) \leq 0$. In addition, the
function $\zeta \mapsto F(|\zeta + C|)$ still does not satisfy the
desired inequality.  Therefore, we define $F$ to be the double
logarithm for large positive values of $\zeta$ and $F \equiv 0$ on
$(-\infty, 0)$. The final challenge is to connect these two regions as
a smooth function that satisfies $\mathcal{L}(F\circ M) \geq 0$.

\subsection{The construction}

Based on the heuristics for the construction of $V_1$ and $V_2$, we
now provide a rigorous proof.

\begin{proof}[Proof of Theorem~\ref{thm:main} (iii)]
  We define $V_1, V_2 : \RR^3 \to \RR$ satisfying the
  hypotheses of Theorem~\ref{thm:criterianon}

\vspace{0.1in}

\textbf{Step 1.}  Fix $R > 1$ such that $H(x, y, z) > 1$ for any
$|(x, y, z)| > R$.  Let $W_2 \in C^2(\RR^3)$ satisfy
\begin{align}
W_2(x, y, z) =
\ln H(x, y, z)  \qquad \text{ for } \qquad |(x,y,z)| >R.  
\end{align}
Then, $W_2 > 0$ outside of a compact set.  Moreover, standard calculations yield
\begin{align}
  \mathcal{L} W_2(x, y, z) 
  =& \frac{2|\beta| z^2 -2 y^2 -2\sigma x^2 -2|\beta| (\sigma+\rho)z
     + 2(\gamma_1 + \gamma_2 +\gamma_3)}{H(x, y, z)}\\
   &- \frac{4(x^2 \gamma_1 + y^2 \gamma_2
     + (z - (\rho + \sigma))^2\gamma_3)}{H^2(x, y, z)} \,.
\end{align}
Consequently, there exists a constant $K>0$ such that 
\begin{equation}\label{eff}
  \mathcal{L} W_2(x, y, z) \leq K  \,\,\,
  \text{ for all } \,\,\, (x,y,z) \in \RR^3. 
\end{equation}
We thus define $V_2=W_2/K$.

\vspace{0.1in}

\textbf{Step 2.}
Define constants
\begin{align}
A = \frac{2\gamma_1 + 2}{|\beta|}, \qquad 
m = \max \{2\gamma_1, 2\sigma^2  \gamma_3\}  \,, 
\end{align}
and let 
\begin{align}
 f(\zeta)  := (1 - \cos \zeta)^2 \,.
\end{align}
One can check that $f(0) = f'(0) = f''(0) = 0$ and $f$ is (strictly)
increasing on $(0, \pi)$, convex on $(0, \frac{2}{3}\pi)$ and concave
on $( \frac{2}{3}\pi, \pi)$. In particular
$f'(\frac{2}{3}\pi) > 0 = f''(\frac{2}{3}\pi)$. By continuity, fix
$B > \frac{2}{3}\pi$ close to $\frac{2}{3}\pi$ such that
$f' \geq -m f''$ on $(\frac{2}{3}\pi, B)$.

Next, for constants $c_0, c_1, c_2$ to be determined in a moment, define 
\begin{align}
\Psi(\zeta) = 
	\begin{cases}
		0 & \zeta  < 0 \,, \\
		(1 - \cos \zeta)^2 = f(\zeta) & \zeta \in [0, B] \,, \\
		c_0 \ln \ln (\zeta + c_1) + c_2 & \zeta > B \,.  
	\end{cases}
\end{align}
We now claim that $c_0, c_1, c_2$ can be chosen such that $\Psi$ is
$C^2$ function.  Because $\Psi$ is $C^2$ function at $0$, we have left
to show that we can find $c_0, c_1, c_2$ such that
\begin{align}
  c_0 \ln \ln (B + c_1) + c_2 &=  f(B) > 0, \\
 \frac{c_0}{(B + c_1) \ln (B + c_1) } &= f'(B) > 0, \\
 -\frac{c_0(1 + \ln(B + c_1))}{[(B + c_1) \ln (B + c_1)]^2 } &= f''(B) < 0 \,.
\end{align}
Substituting the second equation into the third one, we obtain
\begin{equation}\label{eqt}
\frac{1 + \ln(B + c_1)}{(B + c_1) \ln (B + c_1)} = - \frac{f''(B)}{f'(B)} > 0 \,.
\end{equation}
However, the function
\begin{align}
z \mapsto \frac{1 + \ln z}{z \ln z}
\end{align}
is positive and decreasing on $(1, \infty)$ with a vertical asymptote
at $z = 1$ and decaying at infinity. Thus, there exists (unique) $c_1$
such that $B + c_1 > 1$ and \eqref{eqt} holds true. Then, for already
fixed $c_1$ we set
\begin{align}
c_0 = f'(B) (B + c_1) \ln (B + c_1)  > 0
\end{align}
and 
\begin{align}
c_2 = f(B) - c_0 \ln \ln (B + c_1) \,.
\end{align}
It now follows that $\Psi$ is $C^2$ with this choice of constants
$c_0, c_1, c_2$.

Finally fix $\lambda \in (0, 1)$ such that 
\begin{align}
 1 \geq \lambda m \frac{(1 + \ln(B + c_1))}{(B + c_1)\ln (B + c_1)} \,
\end{align}
and define $V_1$ by 
\begin{align}
V_1(x, y, z) = \Psi(\lambda(2\sigma z - x^2 - A)) \
\end{align}
and note that $V_1$ is  $C^2$ function 
and
\begin{equation}\label{ica}
  \mathcal{L} V_1
  = (2\sigma|\beta| z + 2x^2 - 2\gamma_1) \lambda  \Psi'
  + (4x^2 \gamma_1 + 4\sigma^2  \gamma_3) \lambda^2 \Psi'' 
\end{equation}
where, for clarity of presentation, we omitted the argument
$(x, y, z)$ of $V_1$, and $\zeta := \lambda(2\sigma z - x^2 - A)$ of
$\Psi$.

\vspace{0.1in}

\textbf{Step 3.}
We claim that 
\begin{equation}\label{fcl}
\mathcal{L} V_1 \geq 0 \,.
\end{equation}
First, if $\zeta \leq 0$, then $\Psi' = \Psi'' = 0$ and \eqref{fcl}
follows.  For the case when $\zeta \geq 0$, note that since
$A = \frac{2\gamma_1 + 2}{|\beta|}$,
$\zeta = \lambda(2\sigma z - x^2 - A) \geq 0$ implies
\begin{align}
2\sigma z  \geq 2\sigma z - x^2 \geq A = \frac{2\gamma_1 + 2}{|\beta|} \,,
\end{align}
and consequently $2\sigma|\beta| z - 2\gamma_1 \geq 2$. Hence, 
\begin{align} \label{eoc}
2\sigma|\beta| z + 2x^2 - 2\gamma_1 \geq 2(x^2 + 1), \qquad 
0 \leq (4x^2 \gamma_1 + 4\sigma^2 \gamma_3) \leq 2m (x^2 + 1) \,.
\end{align}
Hence, if $\zeta \geq 0$, the coefficients of $\Psi', \Psi''$ in
\eqref{ica} are non-negative.  We split the domain $\zeta \geq 0$ into
three pieces and then finally conclude \eqref{fcl}.

If $\zeta \in [0, \frac{2}{3}\pi]$, then
$\Psi'(\zeta), \Psi''(\zeta) \geq 0$, and the non-negativity of
coefficients of $\Psi', \Psi''$ in \eqref{ica} implies \eqref{fcl}.

If $\zeta \in (\frac{2}{3}\pi, B)$, then $\Psi' (\zeta)> 0$ and
$\Psi''(\zeta) < 0$. Thus, from \eqref{fcl} and \eqref{eoc} follows
\begin{align}
  \frac{1}{\lambda}\mathcal{L} V_1 
     &\geq (2\sigma|\beta| z + 2x^2 - 2\gamma_1)  \Psi' 
       + \lambda(4x^2 \gamma_1 + 4\sigma^2 \gamma_3)  \Psi'' 
       \notag\\
     &\geq 2(x^2 + 1) \Psi' + 2\lambda m (x^2 + 1) \Psi'' \geq 0\,,
       \label{scl}
\end{align}
where in the last inequality we used the definition of $B$ and
$\lambda \in (0, 1]$.

Finally, if $\zeta \in [B, \infty)$, then
$\Psi(\zeta) = c_0 \ln\ln(\zeta + c_1) + c_2$.  Since $c_0 > 0$, one
has $\Psi'(\zeta) > 0$, $\Psi''(\zeta) < 0$.  Using \eqref{scl} and
the fact that the function $z \mapsto \frac{1 + z}{z \ln z}$
decreases, we obtain for any $\zeta > B$
\begin{align}
\frac{1}{\lambda}\mathcal{L} V_1 
&\geq 2(x^2 + 1) \Psi' + 2\lambda m (x^2 + 1) \Psi'' \\
&\geq \frac{2c_0(x^2 + 1)}{(2\zeta + c_1)\ln (\zeta + c_1)}
\left( 1 - \lambda m \frac{(1 + \ln(\zeta + c_1))}{(\zeta + c_1)\ln (\zeta + c_1)}
\right)  \\
&\geq 
\frac{2c_0(x^2 + 1)}{(2\zeta + c_1)\ln (\zeta + c_1)}
\left( 1 - \lambda m \frac{(1 + \ln(B + c_1))}{(B + c_1)\ln (B + c_1)}
\right) \geq 0
\,,
\end{align}
where in the last estimate we used the definition of $\lambda$.
Thus, $\mathcal{L} W_1  \geq 0$ as desired. 
 
\vspace{0.1in}

\textbf{Step 4.}  Let us verify that the assumptions of 
\cref{thm:criterianon} are satisfied with $V_1$ and $V_2 $ .  First
(p4) follows from the construction.  To verify (p1), observe that
\begin{align}
  \limsup_{|(x, y, z)| \to \infty} V_1(x, y, z) 
  &\geq \lim_{z \to \infty} V_1(0, 0, z) 
    =  \lim_{z \to \infty} \Psi (\lambda(2\sigma z - A))  
    \\
   &= \lim_{z \to \infty}  c_0\ln \ln (\lambda(2\sigma z - A) + c_1) + c_2
   =  \infty \,.
\end{align}
Also, $\lim_{|(x, y, z)| \to \infty} H(x, y, z) = \infty$ and (p2) is
satisfied.  Finally, (p3) follows from
\begin{align}
  \limsup_{R \to \infty} 
      \frac{\sup_{|x, y, z| = R} 
       V_1(x, y, z)}{\inf_{|x, y, z| = R} V_2(x, y, z)} 
    &\leq
      \limsup_{R \to \infty} 
         \frac{ V_1(0, 0, R)}{\ln (R^2 - 2(\sigma + \rho)R)} \\
    &\leq \lim_{R \to \infty} 
      \frac{c_0 \ln\ln (\lambda(2\sigma R - A) + c_1) + c_2}
           {\ln (R^2 - 2(\sigma + \rho)R)} 
      = 0 \,,
\end{align}
where we used that $z \mapsto V_1(x, y, z)$ is increasing for large
$z$ and $(x, y) \mapsto V_1(x, y, z)$ is non-increasing.  This
finishes the proof.

\end{proof}

\section*{Acknowledgements}

We would like to thank Centra International de Rencontres
Mathématiques (CIRM), in Marseille, France and the Centro
Internazionale per la Ricerca Matematica (CIRM) in Trento, Itally for
supporting the first two authors with two short visiting research
fellowships where much of this work was conceived.  We would also like
to acknowledge our colleagues Jonathan Mattingly and Jared Whitehead
for helpful feedback on this work.

Our efforts were partially supported under grants NSF-DMS-1816408
(JF), DMS-1313272 (NEGH), DMS-1816551 (NEGH), DMS-1612898 (DPH),
DMS-1855504 (DPH), from the National Science Foundation as well as the
Simons Foundation travel grant 515990 (NEGH).

\begin{footnotesize}
\newcommand{\etalchar}[1]{$^{#1}$}

\end{footnotesize}

\vspace{.3in}
\begin{multicols}{2}

\noindent
Juraj F\"oldes\\ {\footnotesize
Department of Mathematics \\ University of Virginia\\ Web:
\url{http://www.people.virginia.edu/~jf8dc/}\\ Email:
\url{foldes@virginia.edu}} \\[.35cm]

\noindent 
Nathan E. Glatt-Holtz\\ {\footnotesize Department of
Mathematics\\ Tulane University\\ Web:
\url{http://www.math.tulane.edu/~negh/}\\ Email:
\url{negh@tulane.edu}} \\[.35cm] 
 \columnbreak

 \noindent 
David P. Herzog\\ {\footnotesize
Department of Mathematics \\ Iowa State University\\ Web:
\url{http://orion.math.iastate.edu/dherzog/}\\ Email:
\url{dherzog@iastate.edu}} \\[.35cm]
 \columnbreak

\end{multicols}

\end{document}